\documentclass[10pt]{article}
\usepackage{geometry}
\usepackage{mythmsub}
\usepackage{tikz-cd}
\usepackage[shortlabels]{enumitem}
\usepackage{titlesec}
\numberwithin{equation}{thm}
\setlist{leftmargin=1.2cm, topsep=0.5em, itemsep=0em}
\begin{document}
\title{Equivariant cd-structures and descent theory}
\author{Doosung Park}
\date{}
\maketitle
\newcommand{\cps}{{{\rm C}({\rm PSh}(\mathscr{S},{\bf Q}))}}
\newcommand{\cs}{{{\rm C}({\rm Sh}(\mathscr{S},{\bf Q}))}}
\newcommand{\dps}{{{\rm D}({\rm PSh}(\mathscr{S},{\bf Q}))}}
\newcommand{\ds}{{{\rm D}({\rm Sh}(\mathscr{S},{\bf Q}))}}
\begin{abstract}
  We construct the equivariant version of cd-structures, and we develop descent theory for topologies comes from equivariant cd-structures. In particular, we reprove several results of Cisinski-D\'eglies on the \'etale descent, qfh-descent, and h-descent. Since the \'etale topos, qfh-topos, and h-topos do not come from usual cd-structures, such results cannot be produced by usual cd-structures. We also apply equivariant cd-structures to study several topologies on the category of noetherian fs log schemes.
\end{abstract}
\titlelabel{\thetitle.\;\;}
\titleformat*{\section}{\center \bf}
\section{Introduction}
\begin{none}\label{0.1}
  In \cite{Voe10a} and \cite{Voe10b}, analogous results of the Brown-Gersten theorem (\cite{BG73}) for the Nisnevich
  topology and cdh-topology are studied by introducing cd-structures. For instance, if we take $P$ as the
  collection of Nisnevich distinguished squares, then we recover the Nisnevich cd-structure. In  \cite[\S 3.3]{CD12}, it is applied to study descents in triangulated categories of motives over schemes.

  However, there are topoi like the \'etale topos, qfh-topos, and h-topos that cannot be obtained by any
  cd-strucutres. In [loc.\ cit], descent theory for the \'etale topology, qfh-topology, and h-topology is discussed with
  equivarient versions of distinguished squares but without cd-structures. The reason why we introduce
  equivariant cd-structures here is to study descent theory for such a topology more systematically. Here is the definition.
\end{none}
\begin{df}\label{0.2}
    An {\it equivariant
  cd-structure} (or ecd-structure for abbreviation) $P$ on a small category $\mathscr{S}$ with an initial object is a collection of pairs $(G,C)$ where $G$
  is a group and $C$ is a commutative diagram
  \begin{equation}\label{5.3.2}\begin{tikzcd}
    X'\arrow[r,"g'"]\arrow[d,"f'"]&X\arrow[d,"f"]\\
    S'\arrow[r,"g"]&S
  \end{tikzcd}\end{equation}
  in $\mathscr{S}$ with $G$-actions on $X$ over $S$ and on $X'$ over $S'$ such that
  \begin{enumerate}[(i)]
    \item $g'$ is $G$-equivariant over $g$,
    \item if $(G,C)\cong (G',C')$, then $(G,C)\in P$ if and only if $(G',C')\in P$.
  \end{enumerate}
  For a pair $(G,C)\in P$, $C$ is called a $P$-distinguished square of group $G$. The $t_P$-topology is the
  topology generated by the $t_\emptyset$-topology (see (\ref{5.1}) for the definition) and the families of morphisms of the form
  \begin{equation}\label{5.3.1}
    \{f,g\}
  \end{equation}
  for $(G,C)\in P$. If $G$ is trivial for any element of $P$, then $P$ is a usual cd-structure defined in
  \cite[2.1]{Voe10a}.
\end{df}
\begin{none}\label{0.3}
  We can define bounded, complete, and regular ecd-structures as in \cite{Voe10a}. Then we obtain the following two theorems generalizing several results in \cite{Voe10a} and \cite[\S 3]{CD12}.
\end{none}
\begin{thm}\label{0.4}
  {\rm (}See {\rm (\ref{2.14}))} Let $P$ be a complete and regular ecd-structure bounded with respect to a density structure $D$ {\rm (}see {\rm (\ref{5.7})} for the definition{\rm )}, and let $F$ be a $t_P$-sheaf of ${\bf Q}$-modules on a small category $\mathscr{S}$ with an initial object. Then
  \[H_{t_P}^n(S,F)=0\]
  for any $S\in {\rm ob}\,\mathscr{S}$ and $n>{\rm dim}_DS$.
\end{thm}
\begin{thm}\label{0.5}
  {\rm (}See {\rm (\ref{3.5}))} Let $\mathscr{S}$ be a small category with an initial object, let $\mathscr{S}^{\rm dia}$ denote the $2$-category of functors from small categories to $\mathscr{S}$, let $\mathscr{T}:\mathscr{S}^{\rm dia}\rightarrow {\rm Tri}$ be a contravariant pseudofunctor satisfying the conditions {\rm
  (i)--(iv)} in {\rm (\ref{3.1})}, let $K$ be an object of $\mathscr{T}(T)$ where $T\in {\rm ob}\,\mathscr{S}$, and
  let $P$ be a bounded, complete, and regular ecd-structure on $\mathscr{S}$. Then the following are equivalent.
  \begin{enumerate}[{\rm (i)}]
    \item For any distinguished square in $\mathscr{S}/T$ of group $G$ of the form {\rm (\ref{5.3.2})}, the
        commutative diagram
    \[\begin{tikzcd}
      p_*p^*K\arrow[r,"ad"]\arrow[d,"ad"]&p_*g_*g^*p^*K\arrow[d,"ad"]\\
      (p_*f_*f^*p^*K)^G\arrow[r,"ad"]&(p_*h_*h^*p^*K)^G
    \end{tikzcd}\]
    in $\mathscr{T}(T)$ is homotopy Cartesian where $p:S\rightarrow T$ is the structural morphism and $h=fg'$.
    \item $K$ satisfies $t_P$-descent.
  \end{enumerate}
\end{thm}
\begin{rmk}\label{0.6}
  Note that we only work for coefficients that are ${\bf Q}$-algebra. The ${\bf Z}$-coefficient is not suitable for the above theorems.
\end{rmk}
\begin{none}\label{0.7}
  We show that that the \'etale topos, qfh-topos, and h-topos come from ecd-structures. This reproves several results in \cite{Voe96} and \cite[\S 3]{CD12} more systematically. The eh-topos also comes from an ecd-structure, so we can apply the above theorems to the eh-topos also, which seems to be new. We also construct several ecd-structures on the category of noetherian fs log schemes, and define strict closed topology, dividing topology, pw-topology, and qw-topology. Then we prove the following theorem.
\end{none}
\begin{thm}\label{0.8}
  {\rm (}See {\rm (\ref{6.14}))} Let $\mathscr{S}$ be a category of finite dimensional noetherian fs log schemes, let $\mathscr{T}:\mathscr{S}^{\rm dia}\rightarrow {\rm Tri}$ be a contravariant pseudofunctor satisfying the conditions {\rm
  (i)--(iv)} in {\rm (\ref{3.1})}, and let $K$ be an object of $\mathscr{T}(T)$ satisfying the dividing descent and strict closed descent where $T$ is an object of
  $\mathscr{S}$. Then $K$ satisfies the pw-descent if and only
  if $K$ satisfies the qw-descent.
\end{thm}
\begin{none}\label{0.9}
  This theorem has an application in the theory of motives over fs log schemes. The original motivation to define ecd-structures was to prove the above theorem since ecd-structures make us be able to prove it.
\end{none}
\begin{none}\label{0.10}
  {\it Organization of the paper.} In Section 2, we give several definitions. In Section 3, we study the notions of bounded, complete, and regular ecd-structures. In Section 4, we study the descent theory for complex of preseaves of {\bf Q}-modules, and this is generalized for derivators in Section 5 using techniques in \cite[\S 3]{CD12}. In Section 6, we apply our theorems to several topologies on the category of noetherian schemes, which reproves several results on the \'etale topology, qfh-topology, and h-topology in \cite{Voe96} and \cite[\S 3]{CD12} In Section 7, we apply our results to several ecd-structures on the category of noetheiran fs log schemes, which will be useful in the theory of motives over fs log schemes.
\end{none}
\begin{none}\label{0.11}
  {\it Acknowledgement.} The author's dissertation contains many of the results of this paper. The author is grateful to Martin Olsson for helpful communications.
\end{none}
\section{Definitions}
\begin{none}\label{5.6}
  Throughout this paper, $\mathscr{S}$ is a small category with an initial object $\emptyset$. In Section 6, $\mathscr{S}$ is the category of finite dimensional noetherian schemes, and in Section 7, $\mathscr{S}$ is the category of finite dimensional noetherian fs log schemes.
\end{none}
\begin{df}\label{5.1}
  Recall from \cite[\S 4.5.3]{Ayo07} that the $t_\emptyset$-topology on $\mathscr{S}$ is the minimal topology
  such that the empty sieve is a covering sieve for the initial object $\emptyset$. Note that a presheaf $F$ on
  $\mathscr{S}$ is a $t_\emptyset$-sheaf if and only if $F(\emptyset)=*$.
\end{df}
\begin{df}\label{5.5}
Here are some definitions we will frequently use.
  \begin{enumerate}[(1)]
    \item Let $\Lambda$ be a ring. We denote by ${\rm Mod}_{\Lambda}$ the category of $\Lambda$-modules.
    \item Let $t$ be a topology on $\mathscr{S}$. For any object $S$ of $\mathscr{S}$, we denote by $\rho(S)$
        (resp.\ $\rho_t(S)$) the presheaf represented by $S$ (resp.\ $t$-sheaf associated with the presheaf
        represented by $S$).
    \item Let $t$ be a topology on $\mathscr{S}$, and let $\Lambda$ be a ring. For any object $S$ of
        $\mathscr{S}$, we denote by $\Lambda(S)$ (resp.\ ${\Lambda}_t(S)$) the presheaf (resp.\ $t$-sheaf) of
        $\Lambda$-module freely generated by $\rho(S)$ (resp.\ $\rho_t(S)$).
    \item Let $t$ be a topology on $\mathscr{S}$, and let $\Lambda$ be a ring. We denote by ${\rm
        PSh}(\mathscr{S})$ (resp.\ ${\rm Sh}_t(\mathscr{S})$) the category of presheaves (resp.\ $t$-sheaves) of
        sets on $\mathscr{S}$, and we denote by ${\rm PSh}(\mathscr{S},\Lambda)$ (resp.\ ${\rm
        Sh}_t(\mathscr{S},\Lambda)$) the category of presheaves (resp.\ $t$-sheaves) of $\Lambda$-modules on
        $\mathscr{S}$. In ${\rm PSh}(\mathscr{S},\Lambda)$ and ${\rm
        Sh}_t(\mathscr{S},\Lambda)$, we denote by ${\bf 1}$ the constant presheaf ($t$-sheaf associated to) $\Lambda$
    \item An $\mathscr{S}$-{\it diagram} is a functor from a small category to $\mathscr{S}$.
    \item Let $\mathscr{X}:I\rightarrow \mathscr{S}$ be an $\mathscr{S}$-diagram. We often write it as $(\mathscr{X},I)$. For $i\in {\rm ob}\,I$, let $\mathscr{X}_i$ denote the image of $i$ in $\mathscr{S}$.
    \item Let $f:(\mathscr{X},I)\rightarrow (\mathscr{Y},J)$ be a morphism of $\mathscr{S}$-diagrams, and let $i$ be an object of $I$. We denote by $f_i$ the induced morphism $\mathscr{X}_i\rightarrow \mathscr{Y}_{f(i)}$ where $f(i)$ denotes the image of $i$ in $J$.
    \item A morphism $f:(\mathscr{X},I)\rightarrow (\mathscr{Y},J)$ is called {\it reduced} if the induced functor $I\rightarrow J$ is an equivalence.
    \item We denote by $\mathscr{S}^{\rm dia}$ the 2-category of $\mathscr{S}$-diagrams.
    \item We denote by ${\bf e}$ the category with only one object and only one morphism that is the identity
        morphism.
    \item Let $F:\mathcal{C}\rightleftarrows \mathcal{D}:G$ be a pair of adjoint functors between categories. We
        denote by
        \[{\rm ad}:{\rm id}\rightarrow GF,\quad {\rm ad'}:FG\rightarrow {\rm id}\]
        the unit and counit respectively.
  \end{enumerate}
\end{df}
\begin{df}\label{5.2}
  Here are some definitions about group actions.
  \begin{enumerate}[(1)]
    \item Let $G$ be a group. We denote by ${\bf e}_G$ the category with only one object $*$ and ${\rm
        Hom}(*,*)=G$.
    \item Let $A$ be a set or an abelian group with an action of a group $G$. Then we denote by $A^G$ the limit
        of the functor ${\bf e}_G\rightarrow {\rm Set}$ given by
        \[*\mapsto A,\quad g\in {\rm Hom}(*,*)\mapsto g:A\rightarrow A.\]
        Here, ${\rm Set}$ denotes the category of sets. Note that $A^G$ is equal to the subset of $A$ fixed by $G$.
    \item Let $t$ be a topology on $\mathscr{S}$, and let $S$ be an object of $\mathscr{S}$ with an action of a
        group $G$. We denote by $\rho(S)_G$ (resp.\ $\rho_t(S)_G$) the colimit of the functor ${\bf
        e}_G\rightarrow {\rm PSh}(\mathscr{S})$ (resp.\ ${\bf e}_G\rightarrow {\rm Sh}_t(\mathscr{S})$) induced
        by the $G$-action. Note that for any $F\in {\rm PSh}(\mathscr{S})$,
        \[{\rm Hom}_{{\rm PSh}(\mathscr{S})}(\rho(S)_G,F)=F(S)^G.\]
        Thus the induced morphism $\rho(S)\rightarrow \rho(S)_G$ is an epmorphism since $F(S)^G\rightarrow F(S)$ is injective.
    \item Let $f:\mathscr{S}'\rightarrow \mathscr{S}$ be a functor of categories, and let $\Lambda$ be a ring. We
        denote by
    \[f_\sharp,\;\;f_*\]
    the left adjoint and right adjoint of $f^*:{\rm D}({\rm PSh}(\mathscr{S},\Lambda))\rightarrow {\rm D}({\rm
    PSh}(\mathscr{S}',\Lambda))$ respectively.
    \item Let $\Lambda$ be a ring, and let $K$ be an object of ${\rm C}({\rm Mod}_{\Lambda})$ with an action of a
        group $G$. Consider it as an object of ${\rm C}({\rm PSh}({\bf e}_G,\Lambda))$. We denote by
        \[K^G\]
        the object $f_*K$ where $f:{\bf e}_G\rightarrow {\bf e}$ is the trivial functor.

        When $G$ is a finite group and $\Lambda$ is a ${\bf Q}$-algebra, we have the following formula for $K^G$
        in \cite[3.3.21]{CD12}. Consider the morphism $p_K:K\rightarrow K$ given by the formula
        \begin{equation}\label{5.2.1}
          p_K=\frac{1}{|G|}\sum_{g\in G}\sigma^g
        \end{equation}
        where $\sigma^g$ denotes the action of $g$ on $K$. The morphism $p_K$ is a projector, and $K^G$ is the
        image of $p_K$. Note that $K^G$ is a direct summand of $K$.
        \item Let $K$ be an object of ${\rm D}({\rm PSh}(\mathscr{S},\Lambda))$, and let $\mathscr{X}:I\rightarrow \mathscr{S}$ be an $\mathscr{S}$-diagram. We denote by $K(\mathscr{X},I)$ the object of ${\rm Ab}$ given by
            \[p_*\mathscr{X}^*K\]
            where $p:I\rightarrow {\bf e}$ is the functor to the trivial category. Note that we have the formula
            \[K(\mathscr{X},I)\cong R\varprojlim_{i\in I}\mathscr{X}_i.\]
  \end{enumerate}
\end{df}
\begin{df}\label{5.7}
  Let $P$ be an ecd-structure on $\mathscr{S}$. As in \cite{Voe10a}, we have the following definition.
  \begin{enumerate}[(1)]
     \item A $P$-{\it simple cover} is a cover that can be obtained by iterating covers of the form
         (\ref{5.3.1}).
     \item A {\it density structure} on $\mathscr{S}$ is a function which assign to any object $S$ of
         $\mathscr{S}$ a sequence $D_0(S),\;D_1(S)\,\ldots$ of family of morphisms to $S$ with the following
         conditions:
        \begin{enumerate}[(i)]
          \item $(\emptyset\rightarrow S)\in D_0(S)$ for all $S$,
          \item isomorphisms belong to $D_i$ for all $i$,
          \item $D_{i+1}\subset D_i$,
          \item if $g:Y\rightarrow X$ is in $D_i(X)$ and $f:X\rightarrow S$ is in $D_i(S)$, then
              $gf:Y\rightarrow S$ is in $D_i(S)$.
        \end{enumerate}
     \item The {\it dimension} of $S\in {\rm ob}\,\mathscr{S}$ (with respect to a density structure $D$) is the
         smallest number $n$ such that every morphism in $D_n(S)$ is an isomorphism. It is denoted by ${\rm
         dim}_DS$.
     \item Let $D$ be a density structure on $\mathscr{S}$. Then $(G,C)\in P$ is called {\it reducing} (with
         respect to $D$) if for any $i\geq 0$, $X_0\in D_{i+1}(X)$, $S_0'\in D_{i+1}(S')$, and $X_0'\in
         D_{i}(X')$, there exist a distinguished square
        \[C_1=\begin{tikzcd}
          X_1'\arrow[r]\arrow[d]&X_1\arrow[d]\\
          S_1'\arrow[r]&S_1
        \end{tikzcd}\]
         of group $G$ and a morphism $(G,C_1)\rightarrow (G,C)$ of distinguished squares of group $G$ such that
         $S_1\in D_{i+1}(S)$ and that
         \[X_1\rightarrow X,\quad S_1'\rightarrow S',\quad X_1'\rightarrow X'\]
         factors through $X_0$, $S_0$, and $X_0'$ respectively.
     \item A morphism of $P$-distinguished squares $C\rightarrow C'$ of group $G$ is a $G$-equivariant morphism of squares.
     \item A morphism of $P$-distinguished squares of group $G$ is called a {\it refinement} if the morphism of
         the commutative diagrams is an isomorphism on the right corner.
  \end{enumerate}
\end{df}
\begin{df}\label{5.4}
  Let $P$ be an ecd-structure on $\mathscr{S}$. As in \cite{Voe10a}, we introduce the notions of {\it bounded},
  {\it complete}, and {\it regular} ecd-structures as follows.
  \begin{enumerate}[(1)]
    \item $P$ is called {\it bounded} (with respect to a density structure $D$) if every element of $P$ has a refinement that is
        reducing with respect to $D$ and that for any object $S$ of $\mathscr{S}$, ${\rm dim}_DS$ is finite.
    \item $P$ is called {\it complete} if any covering sieve of an object $X\neq \emptyset$ of $\mathscr{S}$
        contains a sieve generated by a $P$-simple cover.
    \item $P$ is called {\it regular} if for any $(G,C)\in P$, $C$ is Cartesian, $S'\rightarrow S$ is a
        monomorphism, and the induced morphism
         \begin{equation}\label{5.4.1}
           (\rho_{t_P}(X')\times_{\rho_{t_P}(S')}\rho_{t_P}(X')_G)\coprod\rho_{t_P}(X)\rightarrow
         \rho_{t_P}(X)\times_{\rho_{t_P}(S)}\rho_{t_P}(X)_G
         \end{equation}
        of $t_P$-sheaves is an epimorphism where $\rho(S)$ denotes the representable $t_P$-sheaf of sets of $S$.
  \end{enumerate}
\end{df}
\section{Properties of ecd-structures}
\begin{none}\label{1.4}
  From (\ref{1.5}) to (\ref{1.7}), we state some results in \cite{Voe10a} that can be trivially generalized to
  ecd-structures. The proofs are identical to those in [loc.\ cit].
\end{none}
\begin{prop}\label{1.5}
  An ecd-structure $P$ on $\mathscr{S}$ is complete if the following two conditions hold.
  \begin{enumerate}[{\rm (i)}]
    \item Any morphism with values in $\emptyset$ is an isomorphism.
    \item For any distinguished square $(G,C)$ of the form {\rm (\ref{5.3.2})} and any morphism $Y\rightarrow X$,
        $(G,C\times_X Y)$ is distinguished.
  \end{enumerate}
\end{prop}
\begin{prop}\label{1.6}
  Let $P_1$ and $P_2$ be ecd-structures on $\mathscr{S}$. If $P_1$ and $P_2$ are complete {\rm (}resp.\ regular,
  resp.\ bounded with respect to the same density structure $D${\rm )}, then $P_1\cup P_2$ is complete {\rm (}resp.\ regular,
  resp.\ bounded with respect to $D${\rm )}.
\end{prop}
\begin{prop}\label{1.7}
  Let $P$ be an ecd-structure on $\mathscr{S}$. For any object $S$ of $\mathscr{S}$, if $P$ is complete {\rm
  (}resp.\ regular, resp.\ bounded with respect to a density structure $D${\rm )}, then $P/S$ is complete {\rm (}resp.\ regular,
  resp.\ bounded with respect to $D/S${\rm )}.
\end{prop}
\begin{prop}\label{1.8}
  Let $P$ be an ecd-structure on $\mathscr{S}$ such that for any $(G,C)\in P$ of the form {\rm (\ref{5.3.2})},
  \begin{enumerate}[{\rm (i)}]
    \item $C$ is Cartesian,
    \item $S'\rightarrow S$ is a monomorphism,
    \item in the Cartesian diagram
    \begin{equation}\label{1.8.1}\begin{tikzcd}
      G\times X'\arrow[d]\arrow[r]&G\times X\arrow[d]\\
      X'\times_{S'}X'\arrow[r,"g'\times_{f'}g'"]&X\times_S X
    \end{tikzcd}\end{equation}
  in $\mathscr{S}$, the family
  \[\{X'\times_{S'}X'\rightarrow X\times_S X,\;G\times X\rightarrow X\times_S X\}\]
  of morphisms is a $t_P$-cover. Here, the right vertical arrow is induced by the projection and the group action $G\times
  X\rightrightarrows X$.
  \end{enumerate}
   Then $P$ is regular.
\end{prop}
\begin{proof}
  We put $t=t_P$. Consider the commutative diagram
  \[\begin{tikzcd}
    \rho_t(X'\times_{S'}X')\coprod \rho_t(G\times X)\arrow[d,"\alpha'\coprod \beta"]\arrow[r,"\gamma'"]&\rho_t(X\times_S X)\arrow[d,"\alpha"]\\
    (\rho_t(X')\times_{\rho_t(S')}\rho_t(X')_G)\coprod \rho_t(X)\arrow[r,"\gamma"]&\rho_t(X)\times_{\rho_t(S)}\rho_t(X)_G
  \end{tikzcd}\]
  of $t_P$-sheaves where $\alpha$, $\alpha'$, and $\beta$ are induced by the $G$-action, and $\gamma$ and $\gamma'$ are induced by (\ref{1.8.1}). The morphism $\rho_t(G\times X)\rightarrow \rho_t(X)$ is an epimorphism since it has a section. Thus morphisms $\rho_t(X)\rightarrow \rho_t(X)_G$ and $\rho_t(X')\rightarrow \rho_t(X')_G$ are also epimorphisms by (\ref{5.2}). Thus the vertical arrows are epimorphisms since
  \[\rho_t(X\times_S X)=\rho_t(X)\times_{\rho_t(S)}\rho_t(X),\quad
  \rho_t(X'\times_{S'}X')=\rho_t(X')\times_{\rho_t(S')}\rho_t(X').\]
  Now the lower horizontal arrow is surjective since the upper horizontal arrow is surjective by (iii).
\end{proof}
\begin{prop}\label{1.1}
  Let $P$ be a regular ecd-structure on $\mathscr{S}$, and let $F$ be a $t_P$-sheaf. Then for a distinguished
  square of the form {\rm (\ref{5.3.2})}, the induced diagram
  \[\begin{tikzcd}
    F(S)\arrow[r]\arrow[d]&F(S')\arrow[d]\\
    F(X)^G\arrow[r]&F(X')^G
  \end{tikzcd}\]
  of sets is Cartesian.
\end{prop}
\begin{proof}
  We put $t=t_P$. Since $F$ is a $t$-sheaf, the induced function $F(S)\rightarrow F(S')\times F(X)$ is injective.
  Thus the induced function
  \[F(S)\rightarrow F(S')\times F(X)^G\]
  is also injective. The remaining is to show that the equalizer of the induced functions
  \begin{equation}\label{1.1.1}
    F(S')\times F(X)^G\rightrightarrows F(X')^G
  \end{equation}
  is $F(S)$. Since $F$ is a $t$-sheaf, $F(S)$ is the equalizer of the induced functions
  \[F(S')\times F(X)\rightrightarrows F(S'\times_S S')\times F(X')\times F(X')\times F(X\times_S X),\]
  which is equal to the equalizer of the induced functions
  \[{\rm Hom}(\rho_t(S')\coprod\rho_t(X),F)\rightrightarrows {\rm Hom}(\rho_t(S'\times_S
  S')\coprod\rho_t(X')\coprod \rho_t(X')\coprod\rho_t(X\times_S X),F).\]
  Note that $S'\times_S S'\cong S'$ since $g$ is a monomorphism. The image of $F(S)$ in $F(S')\times F(X)$ is in
  $F(S')\times F(X)^G$, so we see that $F(S)$ is the equalizer of the induced functions
  \[{\rm Hom}(\rho_t(S')\coprod \rho_t(X)_G,F)\rightrightarrows {\rm Hom}(\rho_t(S')\coprod \rho_t(X')_G\coprod
  \rho_t(X')_G\coprod(\rho_t(X)\times_{\rho_t(S)}\rho_t(X)_G),F).\]
  Thus $F(S)$ is the equalizer of the induced functions
  \[{\rm Hom}(\rho_t(S')\coprod \rho_t(X)_G,F)\rightrightarrows {\rm Hom}( \rho_t(X')_G\coprod
  (\rho_t(X)\times_{\rho_t(S)}\rho_t(X)_G),F).\]
  Using the fact that (\ref{5.4.1}) is surjective, we see that $F(S)$ is the equalizer of the induced functions
  \begin{equation}\label{1.1.3}
    {\rm Hom}(\rho_t(S')\coprod \rho_t(X)_G,F)\rightrightarrows {\rm Hom}( \rho_t(X')_G\coprod
    (\rho_t(X')\times_{\rho_t(S')}\rho_t(X')_G)\coprod \rho_t(X),F).
  \end{equation}
  Now, let $(a,b)\in F(S')\times F(X)^G$ be an element of the equalizer of the functions in (\ref{1.1.1}). Then the
  images of $b$ for the two induced functions
  \begin{equation}\label{1.1.2}
    {\rm Hom}(\rho_t(X)_G,F)\rightrightarrows {\rm Hom}(\rho_t(X')\times_{\rho_t(S')}\rho_t(X')_G,F)
  \end{equation}
  are equal to the images of $a$ for the two equal induced functions
  \[{\rm Hom}(\rho_t(S'),F)\rightrightarrows {\rm Hom}(\rho_t(X')\times_{\rho_t(S')}\rho_t(X')_G,F)\]
  respectively. Thus $b$ is in the equalizer of the functions in (\ref{1.1.2}). This implies that $(a,b)$ is in the equalizer of
  the functions in (\ref{1.1.3}), which completes the proof.
\end{proof}
\begin{prop}\label{1.2}
  Let $P$ be a regular ecd-structure on $\mathscr{S}$. Then for a distinguished square of the form {\rm
  (\ref{5.3.2})}, the induced diagram
  \[\begin{tikzcd}
    \rho_{t_P}(X')_G\arrow[r]\arrow[d]&\rho_{t_P}(X)_G\arrow[d]\\
    \rho_{t_P}(S')\arrow[r]&\rho_{t_P}(S)
  \end{tikzcd}\]
  of $t_P$-sheaves of sets on $\mathscr{S}$ is coCartesian.
\end{prop}
\begin{proof}
  It suffices to show that for any $t_P$-sheaf, the induced diagram
    \[\begin{tikzcd}
    {\rm Hom}(\rho_{t_P}(S),F)\arrow[r]\arrow[d]&{\rm Hom}(\rho_{t_P}(S'),F)\arrow[d]\\
    {\rm Hom}(\rho_{t_P}(X)_G,F)\arrow[r]&{\rm Hom}(\rho_{t_P}(X')_G,F)
  \end{tikzcd}\]
  of sets is Cartesian. It follows from (\ref{1.1}).
\end{proof}
\begin{prop}\label{1.3}
  Let $P$ be a regular ecd-structure on $\mathscr{S}$, and let $\Lambda$ be a ring. Then for a distinguished square of the form {\rm
  (\ref{5.3.2})}, the sequence
  \[0\longrightarrow \Lambda_{t}(X')_G\longrightarrow \Lambda_{t}(X)_G\oplus \Lambda_t(S')\longrightarrow \Lambda_t(S)\longrightarrow 0\]
  in ${\rm Sh}_{t_P}(\mathscr{S},\Lambda)$ is exact.
\end{prop}
\begin{proof}
  The functor $F\mapsto \Lambda(F)$ preserves colimits and monomorphisms. Since $P$ is regular, $X'\rightarrow X$
  is a monomorphism, so the second arrow is a monomorphism. The other part follows from (\ref{1.2}) using the fact
  that the functor preserves colimits.
\end{proof}
\begin{df}\label{2.12}
  Let $\Lambda$ be a ring. For a topology $t$ on $\mathscr{S}$, we have the sheafification functor ${\rm PSh}(\mathscr{S},\Lambda)\rightarrow {\rm Sh}(\mathscr{S},\Lambda)$ and the inclusion functor ${\rm Sh}(\mathscr{S},\Lambda):\rightarrow {\rm PSh}(\mathscr{S},\Lambda)$. We denote by
  \[a_t^*:{\rm D}({\rm PSh}(\mathscr{S},\Lambda))\rightleftarrows {\rm D}({\rm Sh}(\mathscr{S},\Lambda)):a_{t*}\]
   their derived functors.

  Let $K$ be an object of ${\rm D}({\rm PSh}(\mathscr{S},\Lambda))$. We say that $K$ is $t$-{\it local} if the
  morphism
  \[K\stackrel{ad}\longrightarrow a_{t*}a_t^*K\]
  is an isomorphism. Note that if $t'$ is another topology on $\mathscr{S}$ such that the $t$-topos is equivalent to the $t'$-topos, then $K$ is $t$-local if and only if it is $t'$-local.
\end{df}
\begin{none}\label{2.17}
  Here, under the notations and hypotheses as above, we give describe $t$-local objects using $t$-hypercovers. By the abelian version of \cite[1.2]{DHI}, we see that an object $K$ of ${\rm D}({\rm PSh}(\mathscr{S},\Lambda)$ is $t$-local if and only if the morphism
  \[K(X)\rightarrow R\varprojlim_{i\in I}K(\mathscr{X}_i)\cong K(\mathscr{X},I)\]
  is an isomorphism for any object $X$ of $\mathscr{S}$ and $t$-hypercover $(\mathscr{X},I)$ of $X$.
\end{none}
\begin{lemma}\label{2.1}
  Let $t$ be a topology on $\mathscr{S}$, and let $K\in {\rm D}({\rm PSh}(\mathscr{S},\Lambda))$ be an element. If
  $K$ is $t$-local, then
  \[K(S)\cong R{\rm  Hom}_{{\rm D}({\rm PSh}(\mathscr{S},\Lambda))}(\Lambda_t(S),K)\]
  for any $S\in {\rm ob}\,\mathscr{S}$.
\end{lemma}
\begin{proof}
  Since $K(S)\cong R{\rm Hom}_{{\rm D}({\rm PSh}(\mathscr{S},\Lambda))}(\Lambda(S),K)$, it follows from the fact that $a_t^*\Lambda(S)\cong a_t^*\Lambda_t(S)$.
\end{proof}
\begin{lemma}\label{2.2}
  Let $\Lambda$ be a ring, let $t$ be a topology on $\mathscr{S}$, and let $K\in {\rm D}({\rm PSh}(\mathscr{S},\Lambda))$ be an element. If
  $K$ is $t$-local, then
  \[K(S)^G\cong R{\rm  Hom}_{{\rm D}({\rm PSh}(\mathscr{S},\Lambda))}(\Lambda_t(X)_G,K)\]
  for any $S\in {\rm ob}\,\mathscr{S}$ with an action of a group $G$.
\end{lemma}
\begin{proof}
  Since $a_t^*(\Lambda(S)_G)\cong a_t^*(\Lambda_t(S)_G)$, we have that
  \[R{\rm  Hom}_{{\rm D}({\rm PSh}(\mathscr{S},\Lambda))}(\Lambda_t(X)_G,K)\cong R{\rm  Hom}_{{\rm D}({\rm PSh}(\mathscr{S},\Lambda))}(\Lambda(X)_G,K).\]
  Consider the functor $f:{\bf e}_G\rightarrow \mathscr{S}$ induced by the $G$-action. Then
  \[\begin{split}
    R{\rm  Hom}_{{\rm D}({\rm PSh}(\mathscr{S},\Lambda))}(\Lambda(X)_G,K)&\cong R{\rm Hom}_{{\rm D}({\rm PSh}(\mathscr{S},\Lambda))}(f_\sharp f^*{\bf 1},K)\\
    &\cong R{\rm Hom}_{{\rm D}({\rm PSh}({\bf e}_G,\Lambda))}({\bf 1},K(S))\\
    &\cong K(S)^G.
  \end{split}\]
  The conclusion follows from these.
\end{proof}
\begin{prop}\label{2.13}
  Let $P$ be a regular ecd-structure on $\mathscr{S}$, and let $K\in {\rm D}({\rm PSh}(\mathscr{S},{\bf Q}))$ be an element. Then for each distinguished square of group $G$ of the form {\rm (\ref{5.3.2})}, there is a
  homomorphism
  \[\partial_{(G,C)}:H_{t_P}^n(X',K)^G\rightarrow H_{t_P}^{n+1}(S,K)\]
  of ${\bf Q}$-modules for any $n$ such that
  \begin{enumerate}[{\rm (i)}]
    \item for any morphism $(G,C_1)\rightarrow (G,C)$ of distinguished squares of group $G$, the diagram
    \[\begin{tikzcd}
      H_{t_P}^n(X',K)^G\arrow[r,"\partial_{(G,C)}"]\arrow[d]&H_{t_P}^n(S,K)\arrow[d]\\
      H_{t_P}^n(X_1',K)^G\arrow[r,"\partial_{(G,C)}"]&H_{t_P}^n(S,K)
    \end{tikzcd}\]
    commutes,
    \item the sequence of ${\bf Q}$-modules
    \[H^n(S,K)\longrightarrow H^n(X,K)^G\oplus H^n(S',K)\longrightarrow
    H^n(X',K)^G\stackrel{\partial_{(G,C)}}\longrightarrow H^{n+1}(S,K)\]
    is exact.
  \end{enumerate}
\end{prop}
\begin{proof}
  We put $K'=a_{t*}a_t^*F$. Then by (\ref{2.1}),
  \[H_t^n(Y,K)\cong H^n(K'(Y))\cong H^n(R{\rm Hom}({\bf Q}_t(Y),K'))\]
   for any object $Y$ of $\mathscr{S}$. When $Y$ has a $G$-action, by (\ref{2.2}),
   \[H_t^n(Y,K)^G\cong H^n(K'(Y))^G\cong H^n(K'(Y)^G)\cong H^n(R{\rm Hom}({\bf Q}_{t}(Y)_G,K')).\]
   Here, the second isomorphism comes from the fact that $K'(Y)^G$ is given by the projector (\ref{5.2.1}). The conclusion follows
   from these and (\ref{1.3}).
\end{proof}
\section{Descent theory for complexes of presheaves}
\begin{none}\label{2.9}
  In (\ref{2.10}) and (\ref{2.4}), we will have the following notations and hypotheses. Let $P$ be a complete ecd-structure on
  $\mathscr{S}$ bounded with respect to a density structure $D$, and we put $t=t_P$. For $n\geq 0$, let $T^n$ be presheaves of abelian groups on $\mathscr{S}$, and for $(G,C)\in P$ of the form {\rm (\ref{5.3.2})}, let $\partial_{(G,C)}:T^n(X')\rightarrow T^{n+1}(S)$
  be functions. Consider the following conditions.
  \begin{enumerate}
    \item[(i)] $\partial_{(G,C)}$ is natural with respect to morphisms of distinguished squares of group $G$.
    \item[(ii)] The sequence
    \[T^{n-1}(X')^G\rightarrow T^n(S)\rightarrow T^{n}(S')\times T^n(X)^G\]
    is exact for any $n\geq 0$ {\rm (}put $T^{-1}=0${\rm )}.
    \item[(iii)] $T^0(\emptyset)=0$.
    \item[(iii)'] $T^n(\emptyset)=0$ for any $n\geq 0$.
    \item[(iv)] $a_{t*}T^n=0$ for any $n\geq 1$.
    \item[(iv)'] $a_{t*}T^n=0$ for any $n\geq 0$.
  \end{enumerate}
\end{none}
\begin{thm}\label{2.10}
  Under the notations and hypotheses of {\rm (\ref{2.9})}, assume {\rm (i), (ii), (iii)}, and {\rm (iv)}. Then for
  any object $S$ of $\mathscr{S}$,
  \[T^n(S)=0\]
  for any $n>{\rm dim}_DS$.
\end{thm}
\begin{proof}
  We repeat the proof of \cite[2.27]{Voe10a} with minor changes. Refining $P$, we may assume that $P$ contains only
  reducing distinguished squares. For $n\geq 0$, consider the following statement.
  \begin{enumerate}
    \item[$(A_n)$] For any $S\in {\rm ob}\,\mathscr{S}$ and $a\in T^n(S)$, there exists $(j:U\rightarrow S)\in
        D_n(S)$ such that $j^*(a)=0$.
  \end{enumerate}
  We only need to prove $(A_n)$ for all $n\geq 0$. Let us prove it by induction.

  Since $(\emptyset\rightarrow X)\in D_0(S)$ and $T^0(\emptyset)=0$, $(A_0)$ holds. When $n>0$, assume $(A_{n-1})$.
  Let $S$ be an object of $\mathscr{S}$, and let $a\in T^n(S)$ be an element. By (iv), there exists a $t_P$-cover
  $\{j_i:U_i\rightarrow S\}_{i\in I}$ such that $T^n(j_i)(a)=0$ for all $i$. Since $P$ is complete, we may assume
  that the $t_P$-cover is a simple cover. Using the induction on the number of distinguished squares used in the
  cover, we may assume that for some $(G,C)\in P$ of the form (\ref{5.3.2}), there exist $X_0\in D_n(X)$ and
  $S_0'\in D_n(S')$ such that the images of $a$ in $T^n(X_0)$ and $T^n(S_0')$ are 0.

  Since $X'\in D_{n-1}(X')$ and $(G,C)$ is reducing, there exist a distinguished square
  \[C_1=
  \begin{tikzcd}
    X_1'\arrow[r]\arrow[d]&X_1\arrow[d]\\
    S_1'\arrow[r]&S_1
  \end{tikzcd}\]
  of group $G$ and a morphism $(G,C_1)\rightarrow (G,C)$ of distinguished squares of group $G$ such that $S_1\in
  D_n(S)$ and that the images of $a$ in $T^n(X_1)$ and $T^n(S_1')$ are $0$. Let $a_1$ be the image of $a$ in
  $T^n(S_1)$. Then by (ii), $a_1=\partial_{(G,C_1)}b_1$ for some $b_1\in T^{n-1}(X_1')^G$. By $(A_{n-1})$, there
  exist $X_2'\in D_{n-1}(X_1')$ such that the image of $b_1$ in $H^{n-1}(X_2')^G$ is $0$.

  Since $X_1\in D_n(X_1)$, $S_1'\in D_n(S_1')$, and $(G,C)$ is reducing, there exists a distinguished square
  \[C_3=
  \begin{tikzcd}
    X_3'\arrow[r]\arrow[d]&X_3\arrow[d]\\
    S_3'\arrow[r]&S_3
  \end{tikzcd}\]
  of group $G$ and a morphism $(G,C_3)\rightarrow (G,C_1)$ of distinguished squares of group $G$ such that $S_3\in
  D_n(S_1)$ and that the image of $b_1$ in $T^{n-1}(X_3')^G$ is $0$. Note that $S_3\in D_n(S)$ since $S_3\in D_n(S_1)$ and $S_1\in D_n(S)$. Then by (i), the image of $a$ in $T^n(S_3)$ is
  $\partial_{(G,C_3)}0=0$, which proves $(A_n)$ since $S_3\in D_n(S)$.
\end{proof}
\begin{thm}\label{2.4}
  Under the notations and hypotheses of {\rm (\ref{2.9})}, assume {\rm (i), (ii), (iii)'}, and {\rm (iv)'}. Then
  $T^n=0$ for any $n\geq 0$.
\end{thm}
\begin{proof}
  We repeat the proof of \cite[3.2]{Voe10a} with minor changes. Refining $P$, we may assume that $P$ contains only
  reducing distinguished squares. For $d\geq 0$, consider the following statement.
  \begin{enumerate}
    \item[($B_d$)] For any $S\in {\rm ob}\,\mathscr{S}$, $n\geq 0$, and $a\in T^n(S)$, there exists
        $(j:U\rightarrow S)\in D_d(S)$ such that $T^n(j)(a)=0$.
  \end{enumerate}
  Let us prove $(B_d)$ by induction on $d$.

  Since $(\emptyset\rightarrow X)\in D_0(S)$ and $T^n(\emptyset)=0$, $(B_0)$ holds. When $d>0$, assume $(B_{d-1})$.
  Let $S$ be an object of $\mathscr{S}$, and let $a\in T^n(S)$ be an element. By (iv)', there exists a $t_P$-cover
  $\{j_i:U_i\rightarrow S\}_{i\in I}$ such that $T^n(j_i)(a)=0$ for all $i$. Since $P$ is complete, we may assume
  that the $t_P$-cover is a simple cover. Using the induction on the number of distinguished squares used in the
  cover, we may assume that for some $(G,C)\in P$ of the form (\ref{5.3.2}), there exist $S_0'\in D_{d}(S')$ and
  $X_0\in D_{d}(X)$ such that the images of $a$ in $T^n(X_0)$ and $T^n(S_0')$ are $0$.

  Since $X'\in D_{d-1}(X')$ and $(G,C)$ is reducing, there exist a distinguished square
  \[C_1=
  \begin{tikzcd}
    X_1'\arrow[r]\arrow[d]&X_1\arrow[d]\\
    S_1'\arrow[r]&S_1
  \end{tikzcd}\]
  of group $G$ and a morphism $(G,C_1)\rightarrow (G,C)$ of distinguished squares of group $G$ such that $S_1\in
  D_d(S)$ and that the images of $a$ in $T^n(X_1)$ and $T^n(S_1')$ are $0$. Let $a_1$ be the image of $a$ in
  $T^n(S_1)$. Then by (ii), $a_1=\partial_{(G,C_1)} b_1$ for some $b_1\in T^{n-1}(X_1')^G$. By $(B_{d-1})$, there
  exists $X_2'\in D_{d-1}(X_1')$ such that the image of $b_1$ in $T^{n-1}(X_2')^G$ is $0$.

  Since $X_1\in D_d(X_1)$, $S_1'\in D_d(S_1')$, and $(G,C)$ is reducing, there exists a distinguished square
  \[C_3=
  \begin{tikzcd}
    X_3'\arrow[r]\arrow[d]&X_3\arrow[d]\\
    S_3'\arrow[r]&S_3
  \end{tikzcd}\]
  of group $G$ and a morphism $(G,C_3)\rightarrow (G,C_1)$ of distinguished squares of group $G$ such that $S_3\in
  D_d(S_1)$ and that the image of $b_1$ in $T^{n-1}(X_3')^G$ is $0$. Note that $S_3\in D_d(S)$ since $S_3\in D_d(S_1)$ and $S_1\in D_d(S)$. Then by (i), the image of $a$ in $T^n(S_3)$ is
  $\partial_{(G,C_3)}0=0$, which proves $(B_{d})$ since $S_3\in D_d(S)$.

  For large $d$, $D_d(S)$ contains only isomorphisms, so this shows that $a=0$. Thus $T^n=0$ for any $n\geq 0$.
\end{proof}
\begin{thm}\label{2.14}
  Let $P$ be a complete and regular ecd-structure bounded with respect to a density structure $D$, and let $F$ be a $t_P$-sheaf of ${\bf Q}$-modules on $\mathscr{S}$.
  Then
  \[H_{t_P}^n(S,F)=0\]
  for any $S\in {\rm ob}\,\mathscr{S}$ and $n>{\rm dim}_DS$.
\end{thm}
\begin{proof}
  We put $T^n=H^n(-,F)$. Then $T^n$ trivially satisfies the condition (iii) of (\ref{2.9}), and $T^n$ satisfies the
  conditions (i) and (ii) of (loc.\ cit) by (\ref{2.13}). Since higher cohomology $t_P$-sheaf is locally trivial for the
  $t_P$-topology, $T^n$ satisfies the condition (iv) of (\ref{2.9}). Then the conclusion follows from
  (\ref{2.10}).
\end{proof}
\begin{thm}\label{2.3}
  Let $K\in \dps$ be an element, and let $P$ be a ecd-structure on $\mathscr{S}$. Consider the following
  conditions.
  \begin{enumerate}[{\rm (i)}]
    \item $K$ is $t_P$-local.
    \item For any distinguished square of group $G$ of the form {\rm (\ref{5.3.2})}, the commutative diagram
    \begin{equation}\label{2.3.1}\begin{tikzcd}
      K(S)\arrow[r]\arrow[d]&K(S')\arrow[d]\\
      K(X)^G\arrow[r]&K(X')^G
    \end{tikzcd}\end{equation}
    in ${\rm D}({\rm Mod}_{{\bf Q}})$ is homotopy Cartesian.
  \end{enumerate}
  If $P$ is regular, then {\rm (i)} implies {\rm (ii)}. If $P$ is bounded, complete, and regular, then {\rm (ii)}
  implies {\rm (i)}.
\end{thm}
\begin{proof}
  We put $t=t_P$. Assume that $P$ is regular. If $K$ is $t$-local, then to prove (ii), by (\ref{2.1}) and
  (\ref{2.2}), we only need to show that the diagram
  \[\begin{tikzcd}
    R{\rm Hom}_\dps ({\bf Q}_t(S),K)\arrow[r]\arrow[d]&R{\rm Hom}_\dps ({\bf Q}_t(S'),K)\arrow[d]\\
    R{\rm Hom}_\dps ({\bf Q}_t(X)_G,K)\arrow[r]&R{\rm Hom}_\dps ({\bf Q}_t(X')_G,K)
  \end{tikzcd}\]
  in ${\rm D({\rm Mod}_{\bf Q})}$ is homotopy Cartesian. It follows from (\ref{1.3}).

  Assume that $P$ is bounded, complete, and regular. If $K$ satisfies the condition (ii), let $K'$ be a cone of the
  morphism
  \[K\stackrel{ad}\longrightarrow a_{t*}a_t^*K.\]
  Since both $K$ and $a_{t*}a_t^*K$ satisfy the condition (ii), $K'$ also satisfies the condition (ii). We put
  $T^n(S)=H^n(K(S))$ for $S\in {\rm ob}\,\mathscr{S}$. Since $K'$ is acyclic for the topology $t$, $T^n$
  satisfies the conditions (iii)' and (iv)' of (\ref{2.9}). As in the proof of (\ref{2.13}), using (\ref{2.3.1}),
  $T^n$ also satisfies the conditions (i) and (ii) of (\ref{2.9}). Thus by (\ref{2.4}), $K'$ is isomorphic to $0$
  in $\dps$. Then $K$ is isomorphic to $a_{t*}a_t^*K$, so $K$ is $t$-local.
\end{proof}
\begin{none}\label{2.15}
  Let $P$ be a bounded and complete but not necessarily a regular ecd-structure on $\mathscr{S}$, and assume that
  for any $P$-distinguished square of the form (\ref{5.3.2}), $g:S'\rightarrow S$ is a monomorphism.
  Consider the ecd-structure $P'$ that consists of distinguished squares of trivial groups
  \begin{equation}\label{2.15.1}\begin{tikzcd}
    G\times X'\arrow[r,"v'"]\arrow[d,"u'"]&G\times X\arrow[d,"u"]\\
    X'\times_{S'}X'\arrow[r,"v"]&X\times_S X
  \end{tikzcd}\end{equation}
  in (\ref{1.8.1}) for any $P$-distinguished squares of group $G$ of the form (\ref{5.3.2}). Note that $P\cup P'$
  is a bounded, complete, and regular ecd-structure by (\ref{1.6}) and (\ref{1.8}). We have a descent theorem about $P\cup P'$ as follows.
\end{none}
\begin{thm}\label{2.16}
  Under the notations and hypotheses of {\rm (\ref{2.15})}, let $K$ be an object of $\dps$. Then the following are
  equivalent.
  \begin{enumerate}[{\rm (i)}]
    \item $K$ is $t_P$-local, and for any $P'$-distinguished square in $\mathscr{S}$ of group $G$ of
        the form {\rm (\ref{2.15.1})}, the commutative diagram
    \[\begin{tikzcd}
      K(X\times_S X)\arrow[d]\arrow[r]&K(X'\times_{S'}X')\arrow[d]\\
      K(G\times X)\arrow[r]&K(G\times X')
    \end{tikzcd}\]
    in ${\rm D}({\rm Mod}_{\bf Q})$ is homotopy Cartesian.
    \item $K$ is $t_{P\cup P'}$-local.
  \end{enumerate}
\end{thm}
\begin{proof}
  The statement (ii)$\Rightarrow$(i) is done by (\ref{2.3}), so the remaining is the inverse direction. Assume (i).
  By (\ref{2.3}), it suffices to show that for any $P$-distinguished square of group $G$ of the form (\ref{5.3.2}),
  the commutative diagram
  \[\begin{tikzcd}
      K(S)\arrow[r]\arrow[d]&K(S')\arrow[d]\\
      K(X)^G\arrow[r]&K(X')^G
    \end{tikzcd}\]
    in ${\rm D}({\rm Mod}_{{\bf Q}})$ is homotopy Cartesian. Let $(\mathscr{X},\Delta)$ be the \v{C}ech hypercover
    associated to $X\rightarrow S$ where $\Delta$ denotes the simplicial category. Consider the \v{C}ech hypercover  $(\mathscr{Y},J)$ associated to the family of morphisms $\{f:X\rightarrow S,\;g:S'\rightarrow S\}$. Then for any object $j$ of $J$, $\mathscr{Y}_j$ is one of
    \[\mathscr{X}_i,\;S'\times_S \mathscr{X}_i,\;S'\]
    where $i\in {\rm ob}\,\Delta$.

    Then since $K$ is $t_P$-local, by (\ref{2.17}), it suffices to show that the commutative diagrams
    \[\begin{tikzcd}
      K(\mathscr{X}_i)\arrow[r]\arrow[d]&K(S'\times_S \mathscr{X}_i)\arrow[d]\\
      K(X\times_S \mathscr{X}_i)^G\arrow[r]&K(X'\times_S \mathscr{X}_i)^G
    \end{tikzcd}\]
    \[ \begin{tikzcd}
          K(S'\times_S \mathscr{X}_i)\arrow[r]\arrow[d]&K(S'\times_S S'\times_S \mathscr{X}_i)\arrow[d]\\
      K(X\times_S S'\times_S \mathscr{X}_i)^G\arrow[r]&K(X'\times_S S'\times_S \mathscr{X}_i)^G
    \end{tikzcd}\]
        \[\begin{tikzcd}
          K(S')\arrow[r]\arrow[d]&K(S'\times_S S')\arrow[d]\\
      K(X\times_S S')^G\arrow[r]&K(X'\times_S S')^G
    \end{tikzcd}
    \]
    in ${\rm D}({\rm Mod}_{{\bf Q}})$ are homotopy Cartesian for any $i$ where $G$ acts on left $X$ and $X'$. The second and third ones are homotopy
    Cartesian since $g:S'\rightarrow S$ is a monomorphism, so it suffices to show that the first one is homotopy
    Cartesian. The first one is the upper square in the diagram
    \[\begin{tikzcd}
      K(\mathscr{X}_i)\arrow[r]\arrow[d]&K(S'\times_S \mathscr{X}_i)\arrow[d]\\
      K(X\times_S \mathscr{X}_{i})^G\arrow[r]\arrow[d]&K(X'\times_S \mathscr{X}_{i})^G\arrow[d]\\
      K(\mathscr{X}_i)\arrow[r]&K(S'\times_S \mathscr{X}_i)
    \end{tikzcd}\]
    in ${\rm D}({\rm Mod}_{{\bf Q}})$ where the lower square comes from (\ref{2.15.1}).The big square is homotopy Cartesian trivially, and the lower square is homotopy Cartesian by the condition (i). Thus the upper square is homotopy     Cartesian.
\end{proof}
\section{Descent theory for derivators}
\begin{none}
  In this section, we use techniques in \cite[\S 3]{CD12} to transform (\ref{2.3}) and (\ref{2.16}) into statements for contravariant pseudofunctors $\mathscr{T}:\mathscr{S}^{\rm dia}\rightarrow {\rm Tri}$ under some conditions.
\end{none}
\begin{df}
  Let $\mathcal{C}$ be an additive category. An object $C$ of $\mathcal{C}$ is called {\it uniquely divisible} if
  for any integer $n\geq 2$, there exists a unique morphism $f_n:C\rightarrow C$ such that $nf_n={\rm id}$.
\end{df}
\begin{none}\label{3.1}
  Let $\mathscr{T}:\mathscr{S}^{\rm dia}\rightarrow {\rm Tri}$ be a contravariant pseudofunctor. For any $1$-morphism $f$ in $\mathscr{S}^{\rm dia}$, we put
  \[f^*:=\mathscr{T}(f).\]
  We will often assume the following conditions.
  \begin{enumerate}[(i)]
    \item For any object $S$ of $\mathscr{S}$, the prederivator
    \[\mathscr{T}(S\times (-)):\mathscr{S}\rightarrow {\rm Tri}\]
    is a triangulated derivator.
    \item For any $1$-morphism $f$ of $\mathscr{S}$-diagrams, the functor $f^*$ has a right adjoint denoted by
        $f_*$.
    \item For any Cartesian diagram
    \[\begin{tikzcd}
      (\mathscr{Y}',J)\arrow[r,"g'"]\arrow[d,"f'"]&(\mathscr{X}',J)\arrow[d,"f"]\\
      (\mathscr{Y},I)\arrow[r,"g"]&(\mathscr{X},I)
    \end{tikzcd}\]
    of $\mathscr{S}$-diagrams where $g$ is reduced and $f_j$ is an isomorphism for any $j\in {\rm ob}\,J$, the
    natural transformation
    \[f^*g_*\stackrel{Ex}\longrightarrow g_*'f'^*\]
    given by
    \[f^*g_*\stackrel{ad}\longrightarrow f^*g_*f_*'f'^*\stackrel{\sim}\longrightarrow f^*f_*g_*'f'^*\stackrel{ad'}\longrightarrow g_*'f'^*\]
    is an isomorphism.
    \item For any $\mathscr{S}$-diagram $(\mathscr{X},I)$, every object of $\mathscr{T}(\mathscr{X},I)$ is
        uniquely divisible.
  \end{enumerate}
  Note that algebraic derivators defined in \cite[2.4.13]{Ayo07} satisfies the conditions (i)--(iii).
\end{none}
\begin{none}\label{3.2}
  Let $\mathscr{T}:\mathscr{S}^{\rm dia}\rightarrow {\rm Tri}$ be a contravariant pseudofunctor satisfying the conditions (i)--(iv) in (\ref{3.1}), and let $X$ be an object of $\mathscr{S}$. For any object $E$ of $\mathscr{T}(X)$, by
  \cite[3.2.15, 3.2.16]{CD12} and the conditions (i) and (iv) in (\ref{3.1}), we have the morphism
  \[R{\rm Hom}(E,-):\mathscr{T}(X\times (-))\rightarrow {\rm D}({\rm PSh}(-,{\bf Q}))\]
  of derivators (the meaning is that it commutes with $u^*$ for any $1$-morphism $u$ in ${\rm dia}$ where ${\rm dia}$ denotes the $2$-category of small categories) commuting with
  $u_*$ also.

  Consider the $\mathscr{S}$-diagram $(\mathscr{X},\mathscr{S}/X)$ where for any morphism $U\rightarrow V$ in
  $\mathscr{S}/X$, the morphism $\mathscr{X}_U\rightarrow \mathscr{X}_V$ is equal to the morphism $U\rightarrow V$.
  Then consider the $1$-morphism
  \[p:(\mathscr{X},\mathscr{S}/X)\rightarrow (X,\mathscr{S}/X)\]
  of $\mathscr{S}$-diagrams mapping $\mathscr{X}_U\rightarrow \mathscr{X}_V$ to ${\rm id}:X\rightarrow X$ for any morphism $U\rightarrow V$
  in $\mathscr{S}/X$. Then we denote by
  \[\Phi_E:\mathscr{T}(X)\rightarrow {\rm D}({\rm PSh}(\mathscr{S}/X,{\bf Q}))\]
  the functor of triangulated categories given by
  \[R{\rm Hom}(E,p_*f^*(-)).\]

  Now, let $g:(\mathscr{Y},J)\rightarrow X$ be a $1$-morphism of $\mathscr{S}$-diagrams. Then we have the commutative
  diagram
  \[\begin{tikzcd}
    &(\mathscr{Y},J)\arrow[ld,"g"']\arrow[r,"p'"]&(X,J)\arrow[d,"q"]\arrow[lld,"h",near start]\\
    X\arrow[r,"f"',leftarrow]&(\mathscr{X},\mathscr{S}/X)\arrow[r,"p"']\arrow[u,"q'",near end, crossing
    over,leftarrow]&(X,\mathscr{S}/X)
  \end{tikzcd}\]
  of $\mathscr{S}$-diagrams where $h$ denotes the $1$-morphism mapping $J$ to the trivial category, and $q$ and $q'$
  denote the $1$-morphisms induced by $g$. By abuse of notation, we also denote by $h:J\rightarrow {\bf e}$ and
  $q:\mathscr{S}/X\rightarrow J$ the corresponding $1$-morphisms of diagrams.

  Consider the diagram
  \[\begin{tikzcd}
    \mathscr{T}(X)\arrow[r,"f^*"]\arrow[rd,"g^*"']&\mathscr{T}(\mathscr{X},\mathscr{S}/X)\arrow[d,"q'^*"]\arrow[r,"p_*"]
    &\mathscr{T}(X,\mathscr{S}/X)\arrow[d,"q^*"]\arrow[rr,"{R{\rm Hom}(E,-)}"]&&{\rm D}({\rm
    PSh}(\mathscr{S}/X,{\bf Q}))\arrow[d,"q^*"]\\
    &\mathscr{T}(\mathscr{Y},J)\arrow[r,"p_*'"]\arrow[rd,"g_*"']&\mathscr{T}(X,J)\arrow[d,"h_*"]\arrow[rr, "{R{\rm
    Hom}(E,-)}"]&&{\rm D}({\rm PSh}(J,{\bf Q}))\arrow[d,"h_*"]\\
    &&\mathscr{T}(X)\arrow[rr,"{R{\rm Hom}(E,-)}"]&&{\rm D}({\rm Mod}_{\bf Q})
  \end{tikzcd}\]
  of triangulated categories. By the condition (ii) in (\ref{3.1}) and the fact that $R{\rm Hom}(E,-)$ is a
  morphism of derivators commuting with $h_*$, the diagram commutes. Then by the commutativity, for any object $K$
  of $\mathscr{T}(X)$, we have the isomorphism
  \[h_*q^*\Phi_E(K)\stackrel{\sim}\longrightarrow R{\rm Hom}(E,g_*g^*K).\]

  We also have that
  \[\Phi_E(K)(\mathscr{Y},J)=h_*q^*\Phi_E(K).\]
  Thus we have the isomorphism
  \[\Phi_E(K)(\mathscr{Y},J)\stackrel{\sim}\longrightarrow R{\rm Hom}(E,g_*g^*K).\]
  This is functorial in the following sense. Let $g':(\mathscr{Y},J')\rightarrow (\mathscr{Y},J)$ be another
  $1$-morphism of $\mathscr{S}$-diagrams. Then the diagram
  \[\begin{tikzcd}
    \Phi_E(K)(\mathscr{Y},J)\arrow[d]\arrow[r,"\sim"]&R{\rm Hom}(E,g_*g^*K)\arrow[d,"ad"]\\
    \Phi_E(K)(\mathscr{Y}',J')\arrow[r,"\sim"]&R{\rm Hom}(E,g_*g_*'g'^*g^*K)
  \end{tikzcd}\]
  commutes where the left vertical arrow is induced by $g'$.
\end{none}
\begin{df}\label{3.3}
  Let $\mathscr{T}:\mathscr{S}^{\rm dia}\rightarrow {\rm Tri}$ be a contravariant pseudofunctor satisfying the conditions
  (ii) in (\ref{3.1}), let $S$ be an object of $\mathscr{S}$, and let $t$ be a topology on $\mathscr{S}$.
  Following \cite[3.2.5]{CD12}, we say that an object $K$ of $\mathscr{T}(S)$ satisfies $t$-{\it descent} if for
  any morphism $f:X\rightarrow S$ in $\mathscr{S}$ and any $t$-hypercover $g:(\mathscr{X},I)\rightarrow
  X$, the morphism
  \[f_*f^*K\stackrel{ad}\longrightarrow f_*g_*g^*f^*K\]
  in $\mathscr{T}(S)$ is an isomorphism.
\end{df}
\begin{none}\label{3.8}
  Consider the contravariant pseudofunctor ${\rm D}({\rm PSh}(-,\Lambda)):\mathscr{S}^{\rm dia}\rightarrow {\rm Tri}$ where $\Lambda$ is a {\bf Q}-algebra. Let $S$ be an object of $\mathscr{S}$, let $t$ be a topology on $\mathscr{S}$, and let $K$ be an object of $\mathscr{T}(S)$. Then $K$ satisfies $t$-descent if and only if for any morphisms $f:X\rightarrow S$ and $h:Y\rightarrow S$ in $\mathscr{S}$ and any $t$-hypercover $g:(\mathscr{X},I)\rightarrow X$, the homomorphism
  \[(f_*f^*K)(Y)\stackrel{ad}\longrightarrow (f_*g_*g^*f^*K)(Y)\]
  in ${\rm D}({\rm Mod}_{\Lambda})$ is an isomorphism. This is equivalent to the statement that the induced homomorphism
  \[K(X\times_S Y)\longrightarrow K((\mathscr{X},I)\times_S Y)\]
  in ${\rm D}({\rm Mod}_{\Lambda})$ is an isomorphism for any $f$, $h$, and $g$. Thus $K$ satisfies $t$-descent if and only if $K$ is $t$-local by (\ref{2.17}).
\end{none}
\begin{none}\label{3.4}
  Let $\mathscr{T}:\mathscr{S}^{\rm dia}\rightarrow {\rm Tri}$ be a contravariant pseudofunctor satisfying the conditions (iii)
  in (\ref{3.1}), and let $f:X\rightarrow S$ be a morphism in $\mathscr{S}$. Assume that we have a $G$-action on
  $X$ equivariant over $S$. Then this gives a morphism $u:(\mathscr{X},{\bf e}_G)\rightarrow S$ where
  \begin{enumerate}[(i)]
    \item ${\bf e}_G$ denotes the category with single object $*$ and ${\rm Hom}(*,*)=G$,
    \item $\mathscr{X}:{\bf e}_G\rightarrow \mathscr{S}$ is an $\mathscr{S}$-diagram that gives the $G$-action on
        $X$.
  \end{enumerate}
  Let $K$ be an object of $\mathscr{T}(S)$. Then we have defined $(f_*f^*K)^G$ in (\ref{5.2}. Note
  that by \cite[3.3.31]{CD12},
  \[u_*u^*K\cong (f_*f^*K)^G.\]
\end{none}
\begin{thm}\label{3.5}
  Let $\mathscr{T}:\mathscr{S}^{\rm dia}\rightarrow {\rm Tri}$ be a contravariant pseudofunctor satisfying the conditions {\rm
  (i)--(iv)} in {\rm (\ref{3.1})}, let $K$ be an object of $\mathscr{T}(T)$ where $T\in {\rm ob}\,\mathscr{S}$, and
  let $P$ be a bounded, complete, and regular ecd-structure on $\mathscr{S}$. Then the following are equivalent.
  \begin{enumerate}[{\rm (i)}]
    \item For any distinguished square in $\mathscr{S}/T$ of group $G$ of the form (\ref{5.3.2}), the
        commutative diagram
    \[\begin{tikzcd}
      p_*p^*K\arrow[r,"ad"]\arrow[d,"ad"]&p_*g_*g^*p^*K\arrow[d,"ad"]\\
      (p_*f_*f^*p^*K)^G\arrow[r,"ad"]&(p_*h_*h^*p^*K)^G
    \end{tikzcd}\]
    in $\mathscr{T}(T)$ is homotopy Cartesian where $p:S\rightarrow T$ is the structural morphism and $h=fg'$.
    \item $K$ satisfies $t_P$-descent.
  \end{enumerate}
\end{thm}
\begin{proof}
  The statement (i) is equivalent to the statement that for any $E\in {\rm ob}(T(U))$, the induced diagram
  \[\begin{tikzcd}
      R{\rm Hom}(E,p_*p^*K)\arrow[r,"ad"]\arrow[d,"ad"]&R{\rm Hom}(E,p_*g_*g^*p^*K)\arrow[d,"ad"]\\
      R{\rm Hom}(E,(p_*f_*f^*p^*K)^G)\arrow[r,"ad"]&R{\rm Hom}(E,(p_*h_*h^*p^*K)^G)
  \end{tikzcd}\]
  in ${\rm D}({\rm Mod}_{\bf Q})$ is homotopy Cartesian. By (\ref{3.2}) and (\ref{3.4}), this is equivalent to the
  statement that for any $E$, the induced diagram
  \[\begin{tikzcd}
    \Phi_E(K)(S)\arrow[r]\arrow[d]&\Phi_E(K)(S')\arrow[d]\\
    \Phi_E(K)(X)^G\arrow[r]&\Phi_E(K)(X')^G
  \end{tikzcd}\]
  in ${\rm D}({\rm Mod}_{\bf Q})$ is homotopy Cartesian. Then by (\ref{2.3}), this is equivalent to the statement
  that for any $E$, $\Phi_E(K)$ is $t_P$-local. The meaning is that for any $E$, any
  morphism $p:S\rightarrow T$ in $\mathscr{S}$, and any $t_P$-hypercover $q:(\mathscr{X},I)\rightarrow S$, the
  induced morphism
  \[\Phi_E(K)(X)\rightarrow \Phi_E(K)(\mathscr{X},I)\]
  in ${\rm D}({\rm Mod}_{\bf Q})$ is an isomorphism. Now by (\ref{3.2}) again, this is equivalent to
  the statement that for any $E$, $p$, and $q$, the morphism
  \[p_*p^*K\stackrel{ad}\longrightarrow p_*q_*q^*p^*K\]
  is an isomorphism, which is the statement (ii).
\end{proof}
\begin{thm}\label{3.6}
  Let $\mathscr{T}:\mathscr{S}^{\rm dia}\rightarrow {\rm Tri}$ be a contravariant pseudofunctor satisfying the conditions {\rm
  (i)--(iv)} in {\rm (\ref{3.1})}, let $K$ be an object of $\mathscr{T}(T)$ where $T\in {\rm ob}\,\mathscr{S}$, and
  let $P$ be a bounded and complete ecd-structure on $\mathscr{S}$. Consider the bounded, complete, and regular ecd-structure $P'$ on $\mathscr{S}$ defined in {\rm (\ref{2.15})}. Then the following are equivalent.
  \begin{enumerate}
    \item[{\rm (i)}] $K$ satisfies $t_P$-descent, and for any distinguished square in $\mathscr{S}/T$ of group $G$ of the form {\rm (\ref{5.3.2})}, the
        commutative diagram
    \[\begin{tikzcd}
      q_*q^*K\arrow[r,"ad"]\arrow[d,"ad"]&q_*v_*v^*q^*K\arrow[d,"ad"]\\
      q_*u_*u^*q^*K\arrow[r,"ad"]&q_*w_*w^*q^*K^G
    \end{tikzcd}\]
    in $\mathscr{T}(T)$ is homotopy Cartesian where $q:X\times_S X\rightarrow T$ is the structural morphism, $u$ and $v$ are the morphisms given in {\rm (\ref{2.15})}, and $w=uv'$ where $v'$ is the morphism given in {\rm (loc.\ cit)}.
    \item $K$ satisfies $t_{P\cup P'}$-descent.
  \end{enumerate}
\end{thm}
\begin{proof}
  We can transform the statement to that of (\ref{2.16}) as in the proof of (\ref{3.5}).
\end{proof}
\begin{eg}\label{3.7}
  Let $\Lambda$ be a {\bf Q}-algebra, and let $t$ be a topology on $\mathscr{S}$. Assume that $\mathscr{S}$ is the category of noetherian schemes. Consider the contravariant pseudofunctor ${\rm DA}_t(X,\Lambda):\mathscr{S}^{\rm dia}\rightarrow {\rm Tri}$ in \cite[5.3.31]{CD12} (it is writtetn as ${\rm D}_{{\bf A}^1,t}(-,\Lambda)$ in [loc.\ cit]). Then it satisfies the conditions {\rm
  (i)--(iii)} in (\ref{3.1}) by \cite[4.5.30]{Ayo07}. It also satisfies (iv) in (\ref{3.1}) since $\Lambda$ is a {\bf Q}-algebra. Thus we can apply (\ref{3.5}) and (\ref{3.6}) to this case.
\end{eg}
\section{Topologies on the category of schemes}
\begin{none}\label{4.1}
  In this section, we will introduce several ecd-structures that will produce topoi like \'etale topos,
  qfh-topos, and h-topos. Throughout this section, $\mathscr{S}$ is the category of finite dimensional noetherian schemes. We denote by $\mathscr{S}^{ex}$ the category of finite dimensional quasi-excellent noetherian schemes.
\end{none}
\begin{df}
  Let $G$ be a finite group. Recall from \cite[3.3.14]{CD12} that a morphism $f:X\rightarrow S$ of schemes is
  called a {\it pseudo-Galois} cover of group $G$ if it is finite surjective and it has a factorization
  \[X\stackrel{g}\rightarrow Y\stackrel{h}\rightarrow S\]
  such that $g$ is a Galois cover of group $G$ and $h$ is radicial.
\end{df}
\begin{df}
  Consider a Cartesian diagram
  \[C=\begin{tikzcd}
    X'\arrow[r,"g'"]\arrow[d,"f'"]&X\arrow[d,"f"]\\
    S'\arrow[r,"g"]&S
  \end{tikzcd}\]
  in $\mathscr{S}$ and a group $G$ acting on $X$ over $S$. We have several ecd-structures as follows.
  \begin{enumerate}[(1)]
    \item Following \cite{Voe10b}, $C$ is called an {\it additive} distinguished square (of trivial $G$) if
        $X'=\emptyset$ and $S=X\amalg S'$.
    \item Following \cite{Voe10b}, $C$ is called a {\it closed} distinguished square (of trivial $G$) if $f$
        and $g$ are closed immersions and $S=f(X)\cup g(S')$.
    \item Following \cite{Voe10b}, $C$ is called a {\it Zariski} distinguished square (of trivial $G$) if $f$
        and $g$ are open immersions and $S=f(X)\cup g(S')$,
    \item Following \cite{Voe10b}, $C$ is called a {\it Nisnevich} distinguished square (of trivial $G$) if $f$
        is \'etale, $g$ is an open immersion, and the morphism $f^{-1}(S-g(S'))\rightarrow S-g(S')$ is an
        isomorphism. Here, $S-g(S')$ is considered with the reduced scheme structure.
    \item Following \cite{Voe10b}, $C$ is called a {\it proper cdh}-distinguished square (of trivial $G$) if
        $f$ is proper, $g$ is a closed immersion, and the morphism $f^{-1}(S-g(S'))\rightarrow S-g(S')$ is an
        isomorphism.
    \item Following \cite{Voe10b}, $C$ is called a {\it cdh}-distinguished square (of trivial $G$) if it is a
        Nisnevich or cdh-distinguished square.
    \item $C$ is called a {\it Galois} distinguished square of group $G$ if $X'=S'=\emptyset$, $f$ is a Galois cover of group $G$.
    \item Following \cite[3.3.15]{CD12}, $C$ is called a {\it pseudo-Galois qfh}-distinguished square of group
        $G$ if $f$ is finite and surjective, $g$ is a nowhere dense closed immersion, and the induced morphism
        \[f^{-1}(S-g(S'))\rightarrow S-g(S')\]
        is a pseudo-Galois cover of group $G$.
  \end{enumerate}
  Then we obtain the additive, closed, Zariski, Nisnevich, cdh, Galois, and pseudo-Galois qfh
  ecd-structures and topologies using the definitions in (\ref{0.2}).
\end{df}
\begin{df}\label{4.21}
  Recall from \cite[2.1]{Gei06} that the {\it eh}-topology is the topology generated by the cdh-topology and the \'etale topology.
\end{df}
\begin{none}\label{4.25}
  In \cite{Voe10b}, the notion of the standard density structure $D_d$ is defined on $\mathscr{S}$.  Let $X$ be an object of $\mathscr{S}$. Then ${\rm dim}\,X={\rm dim}_DX$. Here, we also recall several lemmas in \cite{Voe10b}.
  \begin{enumerate}[(1)]
    \item If $U,V\in D_d(X)$, then $U\cap V,\in D_d(X)$.
    \item If $U\in D_d(X)$ and $V$ is an open subscheme of $X$, then $U\cap V\in D_d(V)$.
    \item Let $f:X\rightarrow Y$ be a morphism of finite type in $\mathscr{S}$, and assume that there exists an open subscheme $U$ of $Y$ such that $f^{-1}(U)$ is dense in $X$ and $f^{-1}(U)\rightarrow U$ has fibers of dimension zero. Then for any $d\geq 0$ and $V\in D_d(X)$, there exists $W\in D_d(Y)$ such that $f^{-1}(W)\subset V$.
  \end{enumerate}
\end{none}
\begin{none}\label{4.9}
  From (\ref{4.22}) to (\ref{4.4}), we relate some topoi obtained by ecd-structures with \'etale topos, qfh-topos, or h-topos.
\end{none}
\begin{prop}\label{4.22}
  Let $P$ be the union of the Nisnevich and Galois ecd-structures. Then the $t_P$-topology on $\mathscr{S}$ is equal to the \'etale topology on $\mathscr{S}$.
\end{prop}
\begin{proof}
  The \'etale topology is finer than the Nisnevich topology and Galois topology, so we only need to prove the other
  inclusion. Let $f:X\rightarrow S$ be an \'etale cover. Then by \cite[3.3.26]{CD12}, $f$ has a refinement that has
  a factorization $X\stackrel{g}\rightarrow Y\stackrel{h}\rightarrow S$ such that $g$ is a finite \'etale cover and
  $h$ is a Nisnevich cover. Refining $g$ further, we may assume that $g$ is a Galois cover. Then $f$ becomes a
  $t_P$-cover.
\end{proof}
\begin{prop}\label{4.23}
  Let $P$ be the union of the cdh and Galois ecd-structures. Then the $t_P$-topology on $\mathscr{S}$ is equal to the eh-topology on $\mathscr{S}$.
\end{prop}
\begin{proof}
  This is a direct consequence of (\ref{4.22}) since the cdh ecd-structure contains the Nisnevich ecd-structure.
\end{proof}
\begin{prop}\label{4.2}
  Let $P$ be the union of the closed and pseudo-Galois qfh ecd-structures. Then the $t_P$-topology on $\mathscr{S}^{ex}$ is generated by the families of finite morphisms in $\mathscr{S}^{ex}$ that are jointly surjective.
\end{prop}
\begin{proof}
  In every $t_P$-covering, morphisms are finite and jointly surjective. Thus we only need to show that every family
  of finite morphisms that are jointly surjective is a $t_P$-covering. Let $f:X\rightarrow S$ be a finite
  surjective morphism. Since $P$ contains the closed ecd-structure, we only need to show that $f$ is a
  $t_P$-cover.\\[5pt]
  (I) {\it Reduction to the case when $X$ is reduced and $S$ is integral and normal.} Let $\{S_i\}_{i\in I}$ be the
  set of irreducible components of $S$. Then the morphism $\amalg S_{i,{\rm red}}\rightarrow S$ is a $t_P$-cover
  since $P$ contains the closed ecd-structure, so to show that $f$ is a $t_P$-cover, it suffices to show that its
  pullback to each $S_{i,{\rm red}}$ is a $t_P$-cover. Thus we reduce to the case when $S$ is integral.

  To show that $f$ is a $t_P$-cover, it suffices to show that the composition $f:X_{\rm red}\rightarrow S$ is a
  $t_P$-cover. Thus we reduce to the case when $X$ is integral.

  Let $g:S'\rightarrow S$ be the normalization of $S$, and let $h:Z\rightarrow S$ be a nowhere dense closed
  immersion such that $g$ is an isomorphism on $S-h(Z)$. Then by definition, the induced morphism $Z\amalg
  S'\rightarrow S$ is a $t_P$-cover. Thus to show that $f$ is a $t_P$-cover, it suffices to show that the
  projections
  \[Z\times_S X\rightarrow Z,\quad S'\times_S X\rightarrow S'\]
  are $t_P$-covers. By noetherian induction, the first one is a $t_P$-cover, so it suffices to show that the second
  one is a $t_P$-cover. Thus we reduce to the case when $S$ is normal.\\[5pt]
  (II) {\it Reduction to the case when $X$ is integral.} Let $\{X_i\}_{i\in I}$ be the irreducible component of
  $X$, and consider its images $\{Y_i\}_{i\in I}$ in $Y$. Then the induced morphism $\amalg Y_i\rightarrow Y$ is a
  $t_P$-cover by definition, so to show that $f$ is a $t_P$-cover, it suffices to show that for each $j\in I$, the
  family $\{X_i\times_Y Y_j\rightarrow Y_j\}_{i\in I}$ is a $t_P$-cover. To show this, it suffices to show that
  $X_i\rightarrow Y_i$ is a $t_P$-cover. Thus we reduce to the case when $X$ is integral.\\[5pt]
  (III) {\it Final step of the proof.} Let $K$ (resp.\ $L$) be the field of functions of $S$ (resp.\ $X$), and let
  $X'$ be the normalization of $X$ in $L'$ where $L'$ is a Galois extension of the inseparable closure of $K$ in
  $L$. Then it suffices to show that $X'\rightarrow S$ is a $t_P$-cover. Then by \cite[3.3.16]{CD12}, there is a
  nowhere dense closed subscheme $Z$ of $S$ such that $Z\amalg X'\rightarrow S$ is a $t_P$-cover. Thus it suffices
  to show that the projections $X'\times_S Z\rightarrow Z$ and $X'\times_S X'\rightarrow X'$ are $t_P$-cover. The
  first one is a $t_P$-cover by noetherian induction, so it suffices to show that the second one is a $t_P$-cover.
  In the proof \cite[3.3.19]{CD12}, it is shown that $X'\times_S X'=\cup_{g\in G}X'$ in our case. Thus $X'\times_S
  X'\rightarrow X'$ is a $t_P$-cover.
\end{proof}
\begin{prop}\label{4.3}
  Let $P$ be the union of closed, pseudo-Galois qfh, and Nisnevich ecd-structures. Then the $t_P$-topology on $\mathscr{S}^{ex}$ is equal to the qfh-topology on $\mathscr{S}^{ex}$.
\end{prop}
\begin{proof}
  It follows from \cite[3.3.27]{CD12} and (\ref{4.2}).
\end{proof}
\begin{prop}\label{4.4}
  Let $S$ be an object of $\mathscr{S}^{ex}$, and let $t$ be the topology on $\mathscr{S}^{ex}$ generated by the cdh-topology and  qfh-topology. Then a presheaf $F$ of sets on $\mathscr{S}^{ex}/S$ is a $t$-sheaf if and only if $F$ is a $h$-sheaf.
\end{prop}
\begin{proof}
  Since h-topology is finer than cdh-topology and qfh-topology, the statement is a necessary condition. For
  sufficiency, by \cite[3.1.9]{Voe96}, we only need to show the sheaf condition for $f:X\rightarrow S$ that can be
  factorized into $p:X\rightarrow Y$ and $q:Y\rightarrow S$ such that $q$ is proper, birational, and surjective.

  Choose a nowhere dense closed immersion $g:S'\rightarrow S$ such that $q$ is an isomorphism on $S-g(S')$. Then
  the induced morphism $X\amalg S'\rightarrow S$ is a $t$-cover. Consider the sequence
  \[F(S)\rightarrow F(X)\rightrightarrows F(X\times_S X)\]
  of sets, and let $a\in F(X)$ be an element whose images in $F(X\times_S X)$ are equal. Let $b$ be the image of
  $a$ in $F(X\times_S S')$. Then the images of $b$ in $F(X\times_S S'\times_S S')$ are equal. Thus there is an
  element $c\in F(S')$ whose image in $F(X\times_S S')$ is $b$ since by noetherian induction, the sheaf condition
  is satisfied for the projection $X\times_S S'\rightarrow S'$. The images of $c$ in $F(S'\times_S S')$ are also
  equal since $S'\rightarrow S$ is a closed immersion, so the images of $(b,c)$ in
  \[F((X\coprod S')\times_S (X\coprod S'))\]
  are equal. We can use the sheaf condition for $X\coprod S'\rightarrow S$ since it is a $t$-cover, and then we see
  that $b$ is in the image of $F(S)\rightarrow F(X)$.
\end{proof}
\begin{rmk}\label{4.5}
  It seems that $h$-topology does not come from an ecd-structure even though $h$-topos comes from an
  ecd-structure.
\end{rmk}
\begin{none}\label{4.10}
  From (\ref{4.12}) to (\ref{4.17}), we prove that some ecd-structures are bounded, complete, or regular.
\end{none}
\begin{prop}\label{4.12}
  The additive, cdh, closed, Nisnevich, and Zariski ecd-structures are complete, regular, and bounded with respect to the
  standard density structure.
\end{prop}
\begin{proof}
  It is \cite[2.2]{Voe10b}.
\end{proof}
\begin{prop}\label{4.11}
  The Galois and pseudo-Galos qfh ecd-structure are complete.
\end{prop}
\begin{proof}
  It follows from (\ref{1.5}).
\end{proof}
\begin{prop}\label{4.6}
  The pseudo-Galois qfh ecd-structure is bounded with respect to the standard density structure.
\end{prop}
\begin{proof}
  Consider a pseudo-Galois qfh-distinguished square $(G,C)$ of the form (\ref{5.3.2}). Let $Y$ be the scheme
  theoretic closure of the open subscheme $f^{-1}(S-g(S'))$ in $X$. Refining $C$, we can replace $X$ by $\cap_{g\in
  G} g(Y)$, so we may assume that $f^{-1}(S-g(S'))$ is dense in $X$.

  Assume that we have $X_0\in D_{i+1}(X)$, $S_0'\in D_{i+1}(S')$, and $X_0'\in D_i(X')$. Then by (\ref{4.25}(2)), we can
  replace $X_0$ (resp.\ $X_0'$) by $\cap_{g\in G}g(X_0)$ (resp.\ $\cap_{g\in G}g(X_0')$), so we may assume that the
  open immersions $X_0\rightarrow X$ and $X_0'\rightarrow X'$ are $G$-equivariant over $S$ and $S'$ respectively.

  By (\ref{4.25}(3)), we can find $S_0\in D_{i+1}(S)$ and $S_1\in D_{i+1}(S)$ such that $f^{-1}(S_0)\subset X_0$ and
  $g^{-1}(S_1)\subset S_0'$. Then by (\ref{4.25}(1)), $S_0\cap S_1\in D_{i+1}(S)$. Thus by (\ref{4.25}(2)), replacing $S$ by $S_0\cap S_1$, we may assume that $X=X_0$ and $S'=S_0'$.
  We put
  \[S_2=S-fg'(X'-X_0').\]
  Consider the pullback square of $C$ along $S_2\rightarrow S$. To show that this satisfies the condition, the
  remaining is to show that $S_2\in D_{i+1}(S)$.

  Applying (\ref{4.25}(3)) to $f$, it suffices to show that $(X-g'(X'-X_0'))\in D_{i+1}(X)$. Since $X_0'\in D_{i}(X_0)$ and
  $g'$ is a closed immersion, it suffices to show that $X-g'(X')$ is dense in $X$. This follows from the assumption
  that $X-g'(X')=f^{-1}(S-g(S'))$ is dense in $X$.
\end{proof}
\begin{prop}\label{4.13}
  The union of the additive and Galois ecd-structures is regular.
\end{prop}
\begin{proof}
  Let $f:X\rightarrow S$ be a Galois cover of group $G$. Then $X\times_S X\cong G\times X$. Thus the conclusion
  follows from this and (\ref{1.8}).
\end{proof}
\begin{prop}\label{4.8}
  The union of the closed and pseudo-Galois qfh ecd-structures is regular.
\end{prop}
\begin{proof}
  The closed ecd-structure is regular by (\ref{4.12}), so it suffices to show that a pseudo-Galois qfh-distinguished square
  of the form (\ref{5.3.2}) satisfies the condition in (\ref{1.8}). By \cite[3.3.18]{CD12}, $X\times_S X$ has the
  closed cover
  \[(S'\times_S S')\cup \bigcup_{g\in G}X.\]
  The conclusion follows from this and (\ref{1.8}).
\end{proof}
\begin{thm}\label{4.17}
  The union of any combination of the additive, closed, Zariski, Nisnevich, cdh, additive + Galois, closed + pseudo-Galois qfh ecd-structures is complete, regular, and bounded with respect to the standard density structure.
\end{thm}
\begin{proof}
  It follows from (\ref{1.6}), (\ref{4.12}), (\ref{4.13}), and (\ref{4.8}).
\end{proof}
\begin{none}\label{4.18}
  In (\ref{4.19}) and (\ref{4.24}), we collect several our works.
\end{none}
\begin{thm}\label{4.19}
  Let $t$ be the \'etale topology on $\mathscr{S}$ (resp.\ one of the eh-topology, qfh-topology, and h-topology on $\mathscr{S}^{ex}$), and let $F$ be a $t$-sheaf of ${\bf Q}$-modules. Then
  \[H_t^n(S,F)=0\]
  for any $S\in {\rm ob}\,\mathscr{S}$ (resp.\ $S\in {\rm ob}\,\mathscr{S}^{ex}$) and $n>{\rm dim}\,S$.
\end{thm}
\begin{proof}
  By (\ref{4.17}), (\ref{4.22}), (\ref{4.23}), (\ref{4.3}), and (\ref{4.4}), the $t$-topos is
  equivalent to the $t_P$-topos for some complete and regular ecd-structure $P$ bounded with respect to the standard density
  structure. Then the conclusion follows from (\ref{2.10}).
\end{proof}
\begin{thm}\label{4.24}
  Let $\mathscr{T}:(\mathscr{S}^{ex})^{\rm dia}\rightarrow {\rm Tri}$ be a contravariant pseudofunctor satisfying the conditions {\rm
  (i)--(iv)} in {\rm (\ref{3.1})}, let $K$ be an object of $\mathscr{T}(T)$ where $T\in {\rm ob}\,\mathscr{S}^{ex}$.
  Then the following are equivalent.
  \begin{enumerate}[{\rm (i)}]
    \item For any Nisnevich or Galois {\rm (}resp.\ cdh or Galois, resp.\
        Nisnevich or pseudo-Galois qfh, resp.\ cdh or pseudo-Galois qfh{\rm )} distinguished squares of group $G$
        in $\mathscr{S}^{ex}/U$ of the form {\rm (\ref{5.3.2})}, the diagram
        \[\begin{tikzcd}
        p_*p^*K\arrow[r,"ad"]\arrow[d,"ad"]&p_*g_*g^*p^*K\arrow[d,"ad"]\\
        (p_*f_*f^*p^*K)^G\arrow[r,"ad"]&(p_*h_*h^*p^*K)^G
        \end{tikzcd}\]
        is homotopy Cartesian where $p:S\rightarrow T$ is the structural morphism and $h=fg'$.
    \item $K$ satisfies \'etale {\rm (}resp.\ eh, resp.\ qfh, resp.\ h{\rm )} descent.
  \end{enumerate}
\end{thm}
\begin{proof}
  It follows from (\ref{4.17}), (\ref{4.22}), (\ref{4.23}), (\ref{4.3}), (\ref{4.4}), and
  (\ref{3.5}).
\end{proof}
\begin{rmk}
   The above statement still holds if we replace $\mathscr{S}^{ex}$ by $\mathscr{S}$ in the \'etale case. In (\ref{4.19}) and (\ref{4.24}), the results for the \'etale topology, qfh-topology, and h-topology are already
   known in \cite[3.4.7]{Voe96}, \cite[3.4.8]{Voe96}, \cite[3.4.6]{Voe96}, \cite[3.3.32]{CD12}, and
   \cite[3.3.37]{CD12}. We unified the various proofs.
\end{rmk}
\section{Topologies on the category of fs log schemes}
\begin{none}\label{6.2}
  Throughout this section, $\mathscr{S}$ is a category of finite dimensional noetherian fs log schemes.
\end{none}
\begin{df}\label{6.1}
  Consider a Cartesian diagram
  \[C=\begin{tikzcd}
    X'\arrow[r,"g'"]\arrow[d,"f'"]&X\arrow[d,"f"]\\
    S'\arrow[r,"g"]&S
  \end{tikzcd}\]
  in $\mathscr{S}$ and a group $G$ acting on $X$ over $S$. We have several ecd-structures as follows.
  \begin{enumerate}[(1)]
    \item $C$ is called an {\it additive} distinguished square (of trivial $G$) if $X'=\emptyset$ and
        $S=X\amalg S'$.
    \item $C$ is called a {\it strict closed} distinguished square (of trivial $G$) if $f$ and $g$ are strict
        closed immersions and $S=f(X)\cup g(S')$.
    \item $C$ is called a {\it Zariski} distinguished square (of trivial $G$) if $f$ and $g$ are open
        immersions and $S=f(X)\cup g(S')$,
    \item $C$ is called a {\it strict Nisnevich} distinguished square (of trivial $G$) if $f$ is strict
        \'etale, $g$ is an open immersion, and the morphism $f^{-1}(S-g(S'))\rightarrow S-g(S')$ is an
        isomorphism. Here, $S-g(S')$ is considered with the reduced scheme structure.
    \item $C$ is called a {\it Galois} distinguished square of group $G$ if $X'=S'=\emptyset$, $f$ is Galois, and
        $G$ is the Galois group of $f$.
    \item $C$ is called a {\it dividing} distinguished square (of trivial $G$) if $X'=S'=\emptyset$ and $f$ is
        a surjective proper log \'etale monomorphism.
    \item $C$ is called a {\it piercing} distinguished square (of trivial $G$) if $C$ is a pullback of the
        Cartesian diagram
    \begin{equation}\label{6.1.1}\begin{tikzcd}
      {\rm pt}_{\bf N}\arrow[d]\arrow[r]&{\bf A}_{\bf N}\arrow[d]\\
      \spec {\bf Z}\arrow[r]&{\bf A}^1
    \end{tikzcd}\end{equation}
    of $\mathscr{S}$-schemes where the lower horizontal arrow is the 0-section and the right vertical arrow is
    the morphism removing the log structure.
    \item $C$ is called a {\it quasi-piercing} distinguished square (of trivial $G$) if $C$ is a strict closed
        distinguished square $C$ is a piercing distinguished square, or $C$ is a pullback of the Cartesian diagram
    \begin{equation}\label{6.1.2}\begin{tikzcd}
      {\rm pt}_{\bf N}\arrow[r]\arrow[d]&{\bf A}_{\bf N}\arrow[d]\\
      {\rm pt}_{{\bf N}^2}\arrow[r]&{\bf A}_{\bf N}\times_{{\bf A}^1}{\bf A}_{\bf N}
    \end{tikzcd}\end{equation}
    where the lower horizontal arrow is the 0-section and the right vertical arrow is the diagonal morphism of
    ${\bf A}_{\bf N}\rightarrow {\bf A}^1$ removing the log structure.
    \item For $n\in {\bf N}^+$, let $\mu_n$ be an $n$-th root of unity. Then $C$ is called a {\it winding}
        distinguished square of group $G$ if $X'=S'=\emptyset$, $f$ is a pullback of the composition
    \[{\bf A}_Q\times {\rm Spec}\,{\bf Z}[\mu_n]\rightarrow {\bf A}_Q\stackrel{{\bf A}_\theta}\rightarrow{\bf
    A}_P\]
    where the first arrow is the projection, $n\in {\bf N}^+$, and $\theta:P\rightarrow Q$ is a Kummer
    homomorphism of fs monoids such that the Galois group of ${\bf A}_Q\times {\rm Spec}\,{\bf Q}[\mu_n]$ over
    ${\bf A}_P\times {\rm Spec}\,{\bf Q}$ exists, and $G$ is the Galois group.
  \end{enumerate}
  Then we obtain the additive, strict closed, Zariski, strict Nisnevich, dividing, piercing, quasi-piercing,
  Galois, and winding ecd-structures and topologies using the definition in (\ref{0.2}).
\end{df}
\begin{df}\label{6.3}
  If $D$ denotes the standard density structure, for any object $X$, we denote by $D_d(X)$ the family
  $D_d(\underline{X})$. It is again called the {\it standard density structure.}
\end{df}
\begin{prop}\label{6.4}
  The additive, strict closed, Zariski, and strict Nisnevich, and piercing ecd-structures are complete, regular, and bounded with respect to the standard density structure.
\end{prop}
\begin{proof}
  The proof is identical to that of (\ref{4.12}).
\end{proof}
\begin{prop}\label{6.5}
  The piercing, quasi-piercing, Galois, and winding ecd-structures are complete.
\end{prop}
\begin{proof}
  It follows from \cite[2.5]{Voe10b}.
\end{proof}
\begin{prop}\label{6.6}
  The quasi-piercing ecd-structure is regular.
\end{prop}
\begin{proof}
  Consider a commutative diagram
  \[C=\begin{tikzcd}
    X'\arrow[r,"g'"]\arrow[d,"f'"]&X\arrow[d,"f"]\\
    S'\arrow[r,"g"]&S
  \end{tikzcd}\]
  of $\mathscr{S}$-schemes. The diagram
  \[\begin{tikzcd}
        \underline{X'}\arrow[r,"\underline{g'}"]\arrow[d,"\underline{f'}"]&\underline{X}\arrow[d,"\underline{f}"]\\
    \underline{S'}\arrow[r,"\underline{g}"]&\underline{S}
  \end{tikzcd}\]
  is a closed distinguished square, so we are done by (\ref{4.12}).
\end{proof}
\begin{prop}\label{6.7}
  The quasi-piercing ecd-structures is bounded with respect to the standard density structure.
\end{prop}
\begin{proof}
  Consider a quasi-piercing distinguished square
  \[C=\begin{tikzcd}
    X'\arrow[r,"g'"]\arrow[d,"f'"]&X\arrow[d,"f"]\\
    S'\arrow[r,"g"]&S
  \end{tikzcd}\]
  of $\mathscr{S}$-schemes. As in the proof of \cite[2.11]{Voe10b}, if we replace $X$ by the scheme-theoretic
  closure of the open subscheme $f^{-1}(S'-S)$, we get another quasi-piercing distinguished square which is a
  refinement of the original one. Then the same proof of \cite[2.12]{Voe10b} can be applied to our situation.
\end{proof}
\begin{prop}\label{6.8}
  The union of the strict closed and winding ecd-structures is regular.
\end{prop}
\begin{proof}
  Let $f:X\rightarrow S$ be a winding cover, which is a pullback of the composition
  \[{\bf A}_Q\times {\rm Spec}\,{\bf Z}[\mu_n]\rightarrow {\bf A}_Q\stackrel{{\bf A}_\theta}\rightarrow{\bf A}_P\]
  where the first arrow is the projection, $n\in {\bf N}^+$, and $\theta:P\rightarrow Q$ is a Kummer homomorphism
  of fs monoids such that the Galois group $G$ of ${\bf A}_Q\times {\rm Spec}\,{\bf Q}[\mu_n]$ over ${\bf
  A}_P\times {\rm Spec}\,{\bf Q}$ exists. We denote by
  \[\varphi(g):Q\oplus {\bf Z}/(n)\rightarrow \varphi:Q\oplus {\bf Z}/(n)\]
  the homomorphism induced by $g$. We have
  \[X\times_S X=\bigcup_{g\in G}X_g\]
  where $X_g$ denotes the graph of the automorphism $X\rightarrow X$ induced by $g\in G$.

  We will show that $X_g$ is a closed subscheme of $X\times_S X$. We put $Q'=Q\oplus {\bf Z}/(n)$. It suffices to
  show that for any $g\in G$, the homomorphism
  \[Q'\oplus_P Q'\rightarrow Q',\quad (a,b)\mapsto a+\varphi(g)(b)\]
  is strict. Composing with the isomorphism
  \[Q'\oplus_P Q'\rightarrow Q'\oplus Q',\quad (a,b)\mapsto (a,\varphi(g^{-1})(b)),\]
  it suffices to show that the summation homomorphism
  \[Q'\oplus_P Q'\rightarrow Q'\]
  is strict. It follows from (\ref{6.9}) below. Thus $X_g$ is a closed subscheme of $X\times_S X$. Now we can use (\ref{1.8}) to get the conclusion.
\end{proof}
\begin{lemma}\label{6.9}
  Let $\theta:P\rightarrow Q$ be a Kummer homomorphism of fs monoids. Then the summation homomorphism
  $\eta:Q\oplus_P Q\rightarrow Q$ is strict.
\end{lemma}
\begin{proof}
  The homomorphism $\overline{\eta}:\overline{Q\oplus_P Q}\rightarrow\overline{Q}$ is surjective, so the remaining is
  to show that $\overline{\eta}$ is injective. Let $q\in Q$ be an element. Since $\theta$ is Kummer, we can choose $n\in {\bf N}^+$ such that $nq\subset \theta(P)$ for any $q\in Q$. Then $n(q,-q)=(nq,0)+(0,-nq)=0$ because $nq\in \theta(P)$. Thus $(q,-q)\in (Q\oplus_P Q)^*$ since $Q\oplus_P Q$ is
  saturated. Let $Q'$ be the submonoid of $Q\oplus_P Q$ generated by elements of the form $(q,-q)$ for $q\in
  Q^{\rm gp}$. Then $Q'\subset (Q\oplus_P Q)^*$, and $(Q\oplus_P Q)/Q'\cong Q$. The injectivity follows from this.
\end{proof}
\begin{prop}\label{6.10}
  The Galois and winding ecd-structures are bounded with respect to the standard density structure.
\end{prop}
\begin{proof}
  If follows from \cite[2.9]{Voe10b}.
\end{proof}
\begin{thm}\label{6.11}
  The union of any combintation of the additive, strict closed, Zariski, strict Nisnevich, quasi-piercing, additive
  + Galois, and strict closed + winding ecd-structures is complete, regular, and bounded with respect to the standard density
  structure.
\end{thm}
\begin{proof}
  It follows from (\ref{4.17}), (\ref{6.4}), (\ref{6.5}), (\ref{6.6}), (\ref{6.7}), (\ref{6.8}), (\ref{6.10}), and
  (\ref{1.6}).
\end{proof}
\begin{df}\label{6.12}
  The topology on $\mathscr{S}$ generated by the strict \'etale, piercing, and winding topologies is called the
  {\it pw}-topology, and the topology on $\mathscr{S}$ generated by the strict \'etale, quasi-piercing, and winding
  topologies is called {\it qw}-topology.
\end{df}
\begin{thm}\label{6.13}
  Let $\mathscr{T}:\mathscr{S}^{\rm dia}\rightarrow {\rm Tri}$ be a contravariant pseudofunctor satisfying the conditions {\rm
  (i)--(iv)} in {\rm (\ref{3.1})}, and let $K$ be an object of $\mathscr{T}(T)$ satisfying the dividing descent and strict closed descent where $T$ is an object of
  $\mathscr{S}$. Then $K$ satisfies the piercing descent if and only
  if $K$ satisfies the quasi-piercing descent.
\end{thm}
\begin{proof}
  Since the quasi-piercing topology is finer than the piercing topology, the
  sufficient condition holds. Let us prove the necessary condition. Consider a piercing distinguished square
  \[C=\begin{tikzcd}
    X'\arrow[r,"g'"]\arrow[d,"f'"]&X\arrow[d,"f"]\\
    S'\arrow[r,"g"]&S
  \end{tikzcd}\]
  in $\mathscr{S}/T$, and consider the Cartesian diagram
  \[C'=\begin{tikzcd}
    X'\arrow[d,"u'"]\arrow[r,"v'"]&X\arrow[d,"u"]\\
    X'\times_{S'}X'\arrow[r,"v"]&X\times_S X
  \end{tikzcd}\]
  in $\mathscr{S}/T$ where $u$ is the diagonal morphism and $v=g'\times_g g'$. Let $q:X\times_S X\rightarrow T$ be the structural morphism, and we put $w=uv'$. Then by (\ref{3.6}), it suffices to show that the commutative diagram
      \begin{equation}\label{6.13.1}\begin{tikzcd}
      q_*q^*K\arrow[r,"ad"]\arrow[d,"ad"]&q_*v_*v^*q^*K\arrow[d,"ad"]\\
      q_*u_*u^*q^*K\arrow[r,"ad"]&q_*w_*w^*q^*K^G
    \end{tikzcd}\end{equation}
  in $\mathscr{T}(T)$ is homotopypp Cartesian.

  The diagram $C'$ is a pullback of (\ref{6.1.2}), which has the decomposition
  \begin{equation}\label{6.13.3}\begin{tikzcd}
    {\rm pt}_{\bf N}\arrow[r]\arrow[d]&{\bf A}_{\bf N}\arrow[d]\arrow[r,"{\rm id}"]&{\bf A}_{\bf N}\arrow[d,"\alpha"]\\
    V'\arrow[d]\arrow[r]&V\arrow[d]\arrow[r]&W\arrow[d,"\beta"]\\
    {\rm pt}_{{\bf N}^2}\arrow[r,"t"]&{\bf A}_{\bf N}\times_{{\bf A}^1}{\bf A}_{\bf N}\arrow[r,"t'"]&{\bf A}_{\bf N}\times {\bf A}_{\bf N}
  \end{tikzcd}\end{equation}]
  where
  \begin{enumerate}[(i)]
    \item each square is Cartesian,
    \item $t$ is the zero section,
    \item $t'$ is the morphism induced by the morphism ${\bf A}^1\rightarrow {\bf Z}$,
    \item $W$ is the fs log scheme that is the gluing of
        \[{\bf A}_{{\bf N}x\oplus {\bf N}(x^{-1}y)},\quad {\bf A}_{{\bf N}y\oplus {\bf N}(y^{-1}x)}\]
        along ${\bf A}_{{\bf N}x\oplus {\bf Z}(x^{-1}y)}$.
    \item $\alpha$ is the morphism ${\bf A}_{\bf N}\rightarrow V$ given by the homomorphisms
    \[{\bf N}x\oplus {\bf N}(x^{-1}y)\rightarrow {\bf N},\quad {\bf N}y\oplus {\bf N}(y^{-1}x)\rightarrow {\bf N}\]
    mapping $x$ to $1$ and $y$ to $1$,
    \item $\beta$ is the proper log \'etale monomorphism $V\rightarrow {\bf A}_{{\bf N}x\oplus {\bf N}y}$ given by the homomorphism
        \[{\bf N}x\oplus {\bf N}y\rightarrow {\bf N}x\oplus {\bf N}(x^{-1}y),\quad {\bf N}x\oplus {\bf N}y\rightarrow {\bf N}y\oplus {\bf N}(y^{-1}x)\]
        mapping $x$ to $x$ and $y$ to $y$.
  \end{enumerate}
  Then $C'$ has the decomposition
  \begin{equation}\label{6.13.2}\begin{tikzcd}
    X'\arrow[r,"v'"]\arrow[d,"b'"]&X\arrow[d,"b"]\\
    Y'\arrow[r,"v''"]\arrow[d,"a'"]&Y\arrow[d,"a"]\\
    X'\times_{S'}X'\arrow[r,"v"]&X\times_S X
  \end{tikzcd}\end{equation}
  that is the pullback of the left part of (\ref{6.13.3}). The upper square is a strict closed distinguished square and $a$ and $a'$ are dividing covers. Since
  $K$ satisfies the dividing descent, the adjunctions
  \[{\rm id}\stackrel{ad}\longrightarrow a_*a^*,\quad {\rm id}\stackrel{ad}\longrightarrow a_*'a'^*\]
  are isomorphisms. Thus to show that (\ref{6.13.1}) is homotopy Cartesian, it suffices to show that the diagram
  \[\begin{tikzcd}
    (qa)_*(qa)^*K\arrow[r,"ad"]\arrow[d,"ad"]&(qa)_*v_*''v''^*(qa)^*K\arrow[d,"ad"]\\
    (qa)_*b_*b^*(qa)^*K\arrow[r,"ad"]&(qa)_*w_*'w'^*(qa)^*K
  \end{tikzcd}\]
  in $\mathscr{T}(T)$ is homotopy Cartesian where $w'=bv'$. It follows from applying (\ref{3.5}) to strict closed ecd-structure.
\end{proof}
\begin{thm}\label{6.14}
  Let $\mathscr{T}:\mathscr{S}^{\rm dia}\rightarrow {\rm Tri}$ be a contravariant pseudofunctor satisfying the conditions {\rm
  (i)--(iv)} in {\rm (\ref{3.1})}, and let $K$ be an object of $\mathscr{T}(T)$ satisfying the dividing descent and strict closed descent where $T$ is an object of
  $\mathscr{S}$. Then $K$ satisfies the pw-descent if and only
  if $K$ satisfies the qw-descent.
\end{thm}
\begin{proof}
  The sufficient condition holds since the qw-topology is finer than the pw-topology, so let us prove the necessary condition. Assume that $K$ satisfies the pw-descent. Then $K$ satisfies the piercing descent, so by (\ref{6.13}), $K$ satisfies the quasi-piercing descent. Now by (\ref{3.5}) and (\ref{6.12}), $K$ satisfies the qw-descent.
\end{proof}
\titleformat*{\section}{\center \scshape }

\end{document}